\documentclass[a4paper,english,sumlimits,reqno]{amsart}

\usepackage[latin1]{inputenc}
\usepackage[T1]{fontenc}
\usepackage{babel}

\usepackage{xargs}

\usepackage{xcolor}
\usepackage{graphicx}

\usepackage{pstricks}
\usepackage{pstricks-add}
\usepackage{pst-node}
\usepackage[inline]{asymptote}%inline for default inline asymptote code

\usepackage[all]{xy}%provides xymatrix
\newcommand{\vcxymatrix}[1]{\vcenter{\xymatrix{#1}}}%provides a xymatrix which vertically centers eqnumbers

\usepackage{fp}

\usepackage[section]{placeins}
\usepackage{float}

\usepackage{enumerate}

\usepackage{hyperref}
\usepackage{url}

\usepackage{amssymb}
\usepackage{amsmath}
\usepackage{amsthm}
\usepackage{mathtools}
\usepackage{exscale}%to scale mathematical symbols
\usepackage{bbm}%blackboard symbols
\usepackage{bm}%provides \bm which makes stuff bold

\usepackage[mathcal]{eucal}%provides \mathscr --> eulerscript
\usepackage{mathrsfs}%provides \mathscr --> eulerscript

\usepackage{stmaryrd}%provides double brackets \llbracket

\newcounter{proof}
{\stepcounter{proof}\begin{proof}}%
{\end{proof}}%
\newcounter{proofstep}[proof]
{\refstepcounter{proofstep}\bigskip\par\noindent%
  \ifthenelse{\isempty{#1}}%if
    {\textbf{Case \theproofstep.}}%then
    {\textbf{#1.}}%else
  \noindent}%
{\par}%
\newcounter{proofcase}[proof]
{\refstepcounter{proofcase}\bigskip\par\noindent%
  \textbf{Case \theproofcase.}%
  \noindent}%
{\par}%

\theoremstyle{plain}
\newtheorem{thm}{Theorem}[section]
\newtheorem*{thm*}{Theorem}

\newtheorem{lem}[thm]{Lemma}
\theoremstyle{definition}

\theoremstyle{remark}

\numberwithin{equation}{section}

\newcommand{\bb}[1]{\ensuremath{\mathbb #1}}%blackboard
\newcommand{\cal}[1]{\ensuremath{\mathcal #1}}%caligraphic euler
\newcommand{\scr}[1]{\ensuremath{\mathscr #1}}%script
\DeclareMathOperator{\union}{\cup}
\DeclareMathOperator*{\Union}{\bigcup}
\DeclareMathOperator{\isect}{\cap}
\DeclareMathOperator*{\Isect}{\bigcap}

\newcommand{\bmo}[1][]{\ensuremath{BMO_{#1}(\delta)}}
\newcommand{\hardy}[1][]{\ensuremath{H_{#1}^1(\delta)}}
\newcommand{\bmos}[1][]{\ensuremath{BMO_{#1}(\delta^2)}}
\newcommand{\hardys}[1][]{\ensuremath{H_{#1}^1(\delta^2)}}

\newcommand{\cc}[1]{\ensuremath{\llbracket #1 \rrbracket}}

\newcommand{\charfun}{\scalebox{1.0}{\ensuremath{\mathbbm 1}}}

\newcommand{\dint}{\scr D}%dyadic intervals
\newcommand{\drec}{\scr R}%dyadic rectangles
\newcommand{\dif}{\ensuremath{\, \mathrm d}}

\DeclareMathOperator{\sign}{sign}
\DeclareMathOperator{\Id}{Id}

\DeclareMathOperator{\spn}{span}

\DeclareMathOperator{\pred}{\pi}

\DeclareMathOperator{\lev}{lev}

\DeclareMathOperator{\lesslex}{<_\ell}
\DeclareMathOperator{\dilesslex}{\prec_\ell}
\DeclareMathOperator{\drless}{\triangleleft\,}
\DeclareMathOperator{\drlesseq}{\trianglelefteq\,}
\newcommand{\drindex}{\ensuremath{\cal O_{\drless}}}
\begin{document}

\title{Localization and projections on bi--parameter BMO}

\author[R. Lechner]{Richard Lechner}
\address{
  Richard Lechner,
  Institute of Analysis,
  Johannes Kepler University Linz,
  Altenberger Strasse 69,
  A-4040 Linz, Austria
}
\email{Richard.Lechner@jku.at}

\author[P.F.X. M\"uller]{Paul F.X. M\"uller}
\address{
  Paul F.X. M\"uller,
  Institute of Analysis,
  Johannes Kepler University Linz,
  Altenberger Strasse 69,
  A-4040 Linz, Austria
}
\email{Paul.Mueller@jku.at}

\date{\today}
\thanks{The research on this paper is supported by the Austrian Science Foundation (FWF)
  Pr.Nr. P23987 P22549 and by the Hausdorff Institute for Mathematics, Bonn}

\begin{abstract}
  We prove that for any operator $T$ on bi--parameter BMO the identity factors through $T$ or
  $\Id - T$.
  Bourgain's localization method provides the conceptual framework of our proof.
  It consists in replacing the factorization problem on the non--separable bi--parameter BMO by
  its localized, finite dimensional counterpart.
  We solve the resulting finite dimensional factorization problems by exploiting the geometry and
  combinatorics of colored dyadic rectangles.
\end{abstract}

\maketitle

%%%%%%%%%%%%%%%% BEGIN INTRODUCTION %%%%%%%%%%%%%%%%
\section{Introduction}\label{s:intro}

\noindent
The dyadic intervals $\dint$ on the unit interval are given by
\begin{equation*}
  \dint = \{ [2^{-j}k, 2^{-j}(k+1)[ \, :\, j,k \in \bb N_0, k \leq 2^{j}-1 \},
\end{equation*}
and the dyadic rectangles $\drec$ on the unit square by
$\drec = \dint \times \dint$.
For any given dyadic interval $I \in \dint$ we define the $L^\infty$ normalized Haar function $h_I$,
to be $+1$ on the left half of $I$ and $-1$ on the right half of $I$.
Given two dyadic intervals $I,J$ we have
\begin{equation*}
  h_{I\times J}(s,t) = h_I(s)\, h_J(t),
  \qquad s,t \in [0,1[.
\end{equation*}
We define the bi--parameter space $\hardys$ to be the completion of
\begin{equation*}
  \spn\{ h_{I\times J}\, :\, I\times J \in \drec \}
\end{equation*}
under the norm
\begin{equation}\label{eq:definition_hardys_norm}
  \|f\|_{\hardys}
  = \int_0^1 \int_0^1 \big(\sum a_{I\times J}^2\, h_{I\times J}^2 \big)^{1/2}\dif s \dif t,
\end{equation}
where is the finite linear combination $f = \sum a_{I\times J} h_{I\times J}$.
The dual of $\hardys$ is denoted $\bmos$.
It consists of bi--parameter functions $f = \sum a_{I\times J} h_{I\times J}$ with
$\|f\|_{\bmos} < \infty$, where
\begin{equation}\label{eq:definition_bmos_norm}
  \|f\|_{\bmos}
  = \Big( \sup_{\Omega \text{ open}} \frac{1}{|\Omega|}
    \sum_{I\times J\subset \Omega} a_{I\times J}^2\, |I\times J|
  \Big)^{1/2}.
\end{equation}
For basic information and background we refer to~\cite{bernard:1979}, \cite{bourgain:1982},
\cite{brossard:1980}, \cite{chang:1983}, \cite{chang-fefferman:1985}, \cite{gundy:1980}
and~\cite{maurey:1980}.

The main result of this paper is the following theorem.
\begin{thm}[\textbf{Main Theorem}]\label{thm:bmo-primary}
  For any operator
  \begin{equation*}
    S\, :\, \bmos\rightarrow \bmos
  \end{equation*}
  the identity on $\bmos$ factors through $H = S$ or $H = \Id-S$, that is
  \begin{equation}\label{intro:eqn:primarity-bmo}
    \vcxymatrix{\bmos \ar[r]^\Id \ar[d]_E & \bmos\\
      \bmos \ar[r]_H & \bmos \ar[u]_P}
    \qquad \|E\|\|P\| \leq C,
  \end{equation}
  where $C>0$ is some universal constant.
\end{thm}

\noindent
As a consequence, $\bmos$ is a primary Banach space.
Recall that a Banach space $X$ is primary if for any projection $S\, :\, X\rightarrow X$ one of the
spaces $S(X)$ or $(\Id_X - S)(X)$ is isomorphic to $X$.
For background on this classical isomorphic invariant concept we refer
to~\cite{lindenstrauss-tzafriri:1977,mueller:2005,wojtaszczyk:1991}.

One cannot directly deduce the result that $\bmos$ is primary from the previously known result that
$\hardys$ is primary~\cite{mueller:1994}.
To see this, we remark that there exists a projection on $\bmos$ that is not weak$^*$ continuous,
and therefore is not the adjoint of an operator on the predual $\hardys$.
Indeed, given a Banach limit $\ell : \ell^\infty\to \mathbb R$ and a collection of disjoint dyadic
rectangles $\{R_1,R_2,\ldots\}$, let us define the rank one projection
$S : \bmos \to \bmos$ by
\begin{equation*}
  S f
  = \ell \big( (\frac{\langle f, h_{R_n} \rangle}{|R_n|})_{n=1}^\infty \big)
  \sum_{j=1}^\infty h_{R_j}.
\end{equation*}
One can easily verify that $S$ is a bounded linear projection that is not weak$^*$ continuous,
since $f_n = \sum_{j\geq n} h_{R_j}$ is a weak$^*$ null sequence and
$S f_n = \sum_{j=1}^\infty h_{R_j}$.

Our proof of the main theorem is based on the localization method introduced by J. Bourgain
in~\cite{bourgain:1983}.
See also~\cite{bourgain_tzafriri:1987, bourgain_tzafriri:1989}
and~\cite{pelczynski_rosenthal:1974_75} for one of the first papers in this direction.
Bourgain's method is particularly useful for treating factorization problems on non--separable
Banach spaces such as $\bmos$.
It aims at replacing~\eqref{intro:eqn:primarity-bmo} by its localized, finite dimensional
counterpart, and in our context it consists of three basic steps.
\begin{enumerate}[(i)]
\item The starting point consists in applying Wojtaszczyk's isomorphism~\cite{wojtaszczyk:1979} to
  the space $\bmos$ and its finite dimensional building blocks
  $\bmos[n] = \spn\{ h_{I\times J}\, :\, I\times J \in \drec,\, |I|,|J|\geq 2^{-n} \}
  \isect \bmos$.
  This gives
  \begin{equation*}
    \bmos \sim \Big( \sum_n \bmos[n] \Big)_\infty,
  \end{equation*}
  where we use the notation $X\sim Y$ to denote that $X$ is isomorphic to $Y$.

\item Reduction to diagonal operators on $\Big( \sum_n \bmos[n] \Big)_\infty$.

\item
  Verification of the following finite dimensional and quantitative factorization problem:
  For any $n \in \bb N$ and $M > 0$ there exists $N=N(n, M)$ such that for any operator
  $T\, :\, \bmos[N]\rightarrow \bmos[N]$ with $\|T\|\leq M$ we have that $H = T$ or $H = \Id-T$
  satisfies
  \begin{equation}\label{intro:eqn:finite-main-result-hardy}
    \vcxymatrix{\bmos[n] \ar[r]^\Id \ar[d]_E & \bmos[n]\\
      \bmos[N] \ar[r]_H & \bmos[N] \ar[u]_P}
    \qquad \|E\|\|P\|\leq C,
  \end{equation}
  where $C>0$ is some universal constant.
\end{enumerate}
The most challenging aspect in connection with the localization method of Bourgain consists in
proving the finite dimensional factorization problem~\eqref{intro:eqn:finite-main-result-hardy}
while simultaneously controlling $N$ in terms of $n$.
The one--parameter factorization problems -- solved in~\cite{mueller:1988} -- are both the model
case and also a special case of our present problem.
See also~\cite{blower:1990, mueller:2005, mueller:2012, wark:2007}.

\subsection*{Organization of the paper}\hfill\\
\noindent
Section~\ref{s:prelims} lists the preliminary theorems, definitions and concepts.
Section~\ref{s:results} states our main technical results (finite dimensional quantitative
factorization and almost--diagonalization theorems).
Section~\ref{s:reduction} contains the proof of the almost--diagonalization theorem.
Section~\ref{s:factorization} restates the main theorem (infinite dimensional factorization) and
gives its detailed proof.

\subsection*{Acknowledgements}\hfill\\
\noindent
We would like to thank the referee for a careful examination of the manuscript and a very helpful
report.
%%%%%%%%%%%%%%%% END INTRODUCTION %%%%%%%%%%%%%%%%

%%%%%%%%%%%%%%%% BEGIN PRELIMINARIES %%%%%%%%%%%%%%%%
\section{Preliminaries}\label{s:prelims}

\subsection*{Basic notation}\hfill\\

\noindent
Here we collect basic notation and definitions.
We refer to~\cite{mueller:2005} for reference.
Recall that $\dint$ denotes the dyadic subintervals of the unit interval.
Let $\pred\, :\, \dint\setminus\{[0,1)\} \rightarrow \dint$ denote the dyadic predecessor map, that
is $\pred(I) = \Isect\{J \in \dint\, :\, J \supsetneq I\}$.
The level $\lev(I)$ of a dyadic interval $I \in \dint$ is defined as $\lev(I) = -\log_2(|I|)$.
The collection $\scr D_j$ of dyadic intervals at level $j$ is given by
$\dint_j = \{I\in \dint\, :\, \lev(I) = j \}$ and we set $\dint^n = \Union_{j \leq n}\dint_j$.
For $n \in \bb N$ we define
\begin{equation*}
  \drec_n = \big\{ I \times J \in \drec\, :\, I,J \in \dint^n \big\}.
\end{equation*}

Given a collection of sets $\scr C$ we define
\begin{equation*}
  \scr C^* = \Union \{C\, :\, C \in \scr C\}.
\end{equation*}
If $A$ is some set, then $\scr C \isect A = \{C\isect A\, :\, C \in \scr C\}$.
The Carleson constant $\cc{\scr A}$ of a collection $\scr A \subset \dint$ is given
by
\begin{equation*}
  \cc{\scr A} = \sup_{I \in \scr A} \sum_{J \in \scr A\isect I} |J|/|I|.
\end{equation*}
Note that $\cc{\scr A \union \scr B} \leq \cc{\scr A} + \cc{\scr B}$ for any two collections
$\scr A, \scr B \subset \dint$.

For any given dyadic interval $I \in \dint$ we define
$h_I = \charfun_{I_0} - \charfun_{I_1}$, where $\charfun_A$ denotes the characteristic function of a
set $A$, $I_0 = [\inf I, (\inf I + \sup I)/2[$ and $I_1 = [(\inf I + \sup I)/2, \sup I[$.
The one parameter hardy space $\hardy$ is the completion of
\begin{equation*}
  \spn\{ h_I\, :\, I \in \dint \}
\end{equation*}
under the square function norm
\begin{equation*}
  \|f\|_{\hardy}
  = \int_0^1 \big(\sum a_I^2\, h_I^2 \big)^{1/2}\dif t,
\end{equation*}
where $f = \sum a_I h_I$.
We set
\begin{align*}
  \hardy[n] &= \spn\{ h_I\, :\, I \in \dint^n \} \isect \hardy,\\
  \bmo[n] &= \spn\{ h_I\, :\, I \in \dint^n \} \isect \bmo.
\end{align*}

\subsection*{The Gamlen--Gaudet factorization}\hfill\\

\noindent
We recall the relation between large Carleson constants and factorization, see~\cite{mueller:2005}.
Let $\scr A$ be a collection of dyadic intervals satisfying $\cc{\scr A} \geq N$ and define
\begin{equation*}
  X_{\scr A} = \spn\{ h_I\, :\, I \in \scr A \} \isect \hardy.
\end{equation*}
If $N = n 4^n$, then there exist linear operators $E$ and $P$ so that
\begin{equation}\label{prelims:eq:condensation-1d}
  \vcxymatrix{
    \hardy[n] \ar[rr]^\Id \ar[dr]_E & & \hardy[n]\\
    & X_{\scr A} \ar[ur]_P &
  }
  \qquad \|E\|\|P\| \leq C.
\end{equation}

The bi--parameter analogues of these finite dimensional building blocks are:
\begin{align*}
  \hardys[n] &= \spn\{ h_{I\times J}\, :\, I\times J \in \drec_n \} \isect \hardys,\\
  \bmos[n] &= \spn\{ h_{I\times J}\, :\, I\times J \in \drec_n \} \isect \bmos.
\end{align*}
In the bi--parameter context, factorization and large Carleson constants are related for collections
of rectangles having product structure.
Given collections of dyadic intervals $\scr A,\, \scr B\subset \dint$ we define the product space
$X_{\scr A\times \scr B}$ by
\begin{equation*}
  X_{\scr A\times \scr B}
  = \spn\{ h_{I\times J}\, :\, I\times J \in \scr A\times \scr B \} \isect \hardys.
\end{equation*}
If $N = n 4^n$ and $\cc{\scr A} \geq N$, $\cc{\scr B} \geq N$, then there exist linear operators
$E$ and $P$ such that
\begin{equation}\label{prelims:eq:condensation-2d}
  \vcxymatrix{
    \hardys[n] \ar[rr]^\Id \ar[dr]_E & & \hardys[n]\\
    & X_{\scr A\times \scr B}\ar[ur]_P &
  }
  \qquad \|E\|\|P\| \leq C.
\end{equation}
Due to product structure of $X_{\scr A\times \scr B}$, the bi--parameter
factorization~\eqref{prelims:eq:condensation-2d} results directly from its one--parameter
predecessor~\eqref{prelims:eq:condensation-1d}.
In the next paragraph we discuss Ramsey's theorem for colored dyadic rectangles.
Its relevance for the constructions of this paper comes from the fact that for any two--coloring of
$\drec_n$, Ramsey's theorem detects a large monochromatic collection of the form
$\scr A\times \scr B$.

\subsection*{Ramsey theorem for colored dyadic rectangles}\hfill\\

\noindent
Ramsey's theorem asserts that for any two--coloring of the dyadic rectangles
\begin{equation*}
  \drec_n = \big\{ I\times J\, :\, |I|\geq 2^{-n},\ |J| \geq 2^{-n} \big\}
\end{equation*}
there exist collections $\scr A$, $\scr B$ of dyadic intervals, each of which has
large Carleson constant and, moreover,
\begin{equation*}
  \text{$\scr A \times \scr B$ is monochromatic in $\drec_n$.}
\end{equation*}
For later reference we state this assertion in the following theorem.
\begin{thm}\label{thm:ramsey}
  Given $n_0 \in \bb N$ there exists $n \in \bb N$ such that for any collection $\scr
  C \subset \drec_n$ one finds $\scr A,\scr B \subset \dint$ satisfying
  \begin{enumerate}[(i)]
  \item $\scr A \times \scr B \subset \scr C$
    or
    $\scr A \times \scr B \subset \drec_n \setminus \scr C$,
  \item $\cc{\scr A} \geq n_0$ and $\cc{B} \geq n_0$.
  \end{enumerate}
  One can choose $n = n_0\, 2^{4^{n_0}}$.
\end{thm}
For the above formulation of Ramsey's theorem we refer to~\cite{mueller:1994}.
See also~\cite[Chapter 1]{graham:rothschild:spencer:1980}.
For convenience, we give the proof here.
\begin{proof}
  Define $n = n_0\, 2^{4^{n_0}}$ and let $\mathscr C \subset \drec_n$.
  Define $k = 2 n_0 -1$ and let $I_1, \ldots, I_{2^{k+1}-1}$ be an enumeration of the dyadic
  intervals in $\dint^k$.
  First, we set $\mathscr E_0 = \mathscr F_0 = \mathscr G_0 = \dint^n$, $I_0 = \emptyset$ and
  $f(I_0) = \infty$.
  Second, assuming that $\mathscr E_j$, $\mathscr F_j$, $\mathscr G_j$ and $f(I_j)$ have already
  been constructed for all $0 \leq j \leq m-1$, we define the collections
  \begin{equation*}
    \mathscr E_m = \{ J \in \mathscr G_{m-1}\, :\, I_m\times J \notin \mathscr C\}
    \quad\text{and}\quad
    \mathscr F_m  = \{ J \in \mathscr G_{m-1}\, :\, I_m\times J \in \mathscr C\}.
  \end{equation*}
  If $\cc{\mathscr E_m} \geq \cc{\mathscr F_m}$, we set $f(I_m) = 0$, and if
  $\cc{\mathscr E_m} < \cc{\mathscr F_m}$, we set $f(I_m) = 1$.
  To conlude the inductive step, we define $\mathscr G_m = \mathscr E_m$ if $f(I_m) = 0$, and
  $\mathscr G_m = \mathscr F_m$ if $f(I_m) = 1$.

  Observe that $\mathscr E_m \cup \mathscr F_m = \mathscr G_{m-1}$ and
  $\cc{\mathscr G_m} = \max(\cc{\mathscr E_m},\cc{\mathscr F_m})$.
  By subadditivity of $\cc{\cdot}$ we obtain $2 \cc{\mathscr G_m} \geq \cc{\mathscr G_{m-1}}$, and
  by iterating
  $\cc{\mathscr G_m} \geq 2^{-m} (n + 1)$.
  For $m = 2^{k+1}-1$, we set $\mathscr B = \mathscr G_{m}$, and
  observe that by our choices $\cc{\mathscr B} \geq 2^{-4^{n_0}+1}(n+1) \geq n_0$.
  Now define
  \begin{equation*}
    \mathscr H_0 = \{ I \in \dint^k\, :\, f(I) = 0 \}
    \qquad\text{and}\qquad
    \mathscr H_1  = \{ I \in \dint^k\, :\, f(I) = 1 \},
  \end{equation*}
  and note that $\mathscr H_0 \cup \mathscr H_1 = \dint^k$.
  By definition of $f$ and $\mathscr G_j$ we have
  \begin{equation*}
    \mathscr H_0\times \mathscr B \subset (\dint^k\times \dint^k)\setminus \mathscr C
    \qquad\text{and}\qquad
    \mathscr H_1\times \mathscr B \subset \mathscr C.
  \end{equation*}
  Now, let $\mathscr A = \mathscr H_0$ if $\cc{\mathscr H_0} \geq \cc{\mathscr H_1}$, and
  $\mathscr A = \mathscr H_1$ if $\cc{\mathscr H_0} < \cc{\mathscr H_1}$.
  We conclude this proof by observing that by our choices we have
  \begin{equation*}
    2 \cc{\mathscr A}
    \geq \cc{\mathscr H_0} + \cc{\mathscr H_1}
    \geq \cc{\mathscr H_0 \cup \mathscr H_1}
    = \cc{\dint^k} = k+1 = 2 n_0.
    \qedhere
  \end{equation*}
\end{proof}

\subsection*{Block bases and projections in $\hardys$}\hfill\\

\noindent
We introduce next some frequently used terminology and record a boundedness criterion for
projections on $\hardys$.
We say that a sequence $ \{b_{I \times J}\, :\, I\times J \in \drec \}$ in a Banach space $X$ is
equivalent to the unconditional 2D Haar basis $\{ h_{I\times J}\, :\, I\times J \in \drec \}$ in
$\hardys$ if the following holds:
The map
\begin{equation*}
  T : \sum a_{I \times J}\, h_{I\times J} \to \sum a_{I \times J}\, b_{I \times J}
\end{equation*}
defined initially on finite linear combinations of 2D Haar functions and extended by density to
$\hardys$ satisfies
\begin{equation*}
  C_1^{-1} \|x\|_{\hardys} \leq \|T(x)\|_X \leq C_1\, \|x\|_{\hardys},
  \quad\quad
  x \in \hardys.
\end{equation*}
Let $\{\scr E_{I\times J } : I\times J \in \drec \}$ be pairwise disjoint collections of dyadic
rectangles and let
$E_{I\times J } = \scr E_{I\times J}^* = \Union\limits_{K\times L \in \scr E_{I\times J}} K\times L$ be
the point-set covered by the collection $\scr E_{I\times J }$.
We denote by
\begin{equation*}
  b_{I\times J } = \sum _{K\times L\in \scr E_{I\times J }} h_{K\times L}
\end{equation*}
the block-basis generated by $\scr E_{I\times J}$.
We assume throughout, that $\|b_{I\times J}\|_2^2 = |E_{I\times J }|$ or equivalently that
$\scr E_{I\times J}$ consists of pairwise disjoint dyadic rectangles.
We formulate conditions on the collections $\{\scr E_{I\times J }\}$ so that the block basis
$\{b_{I\times J } \}$ is equivalent to the 2D Haar system.
The sets  $\{E_{I\times J } : I\times J \in \drec \} $ satisfy the \textbf{bi-tree condition} if
there exists $C_2 > 0 $ so that for each $I\times J \in \drec$
\begin{subequations}\label{eq:bi-tree}
\begin{equation}\label{eq:bi-tree:a}
  C_2^{-1} | I \times J | \leq |E_{I\times J}| \leq C_2\, |{I\times J }|.
\end{equation}
and for $(I_0 \times J_0), (I_1 \times J_1) \in \drec$ with $I = \widetilde I_0 = \widetilde I_1$
and $J = \widetilde J_0 = \widetilde J_1$ we have
\begin{align}
  E_{I_0\times J } \isect E_{I_1\times J} & = \emptyset,
  & E_{I_0\times J }\cup E_{I_1 \times J} & \subset E_{I\times J},
  \label{eq:bi-tree:b}\\
  E_{I\times J_0 } \isect E_{I\times J_1 } & = \emptyset,
  & E_{I\times J_0 }\cup E_{I \times J_1} & \subset E_{I\times J}.
  \label{eq:bi-tree:c}
\end{align}  
\end{subequations}
If~\eqref{eq:bi-tree} is satisfied, then the block basis
$\{b_{I\times J }\, :\, I\times J\in \drec\}$ is equivalent to the 2D Haar system in$\hardys$ and
$C_1 = C_1(C_2)$.
The following theorem is a basic tool that allows to project onto the span of the block bases
$\{b_{I\times J }  : I\times J \in \drec\}$.
It was instrumental in proving that $\hardys$ is a primary space, see~\cite{capon:1982}
and~\cite{mueller:1994}.
In the present paper, the main component of the factoring operator $P$ appearing in
Theorem~\ref{thm:bmo-primary} consists of a weighted version of the following orthogonal projection
$Q$ onto block basis.

\begin{thm}\label{pro:projection}
  Let $\scr E_{I\times J }$, $I\times J \in \drec$ be pairwise disjoint collections consisting of
  disjoint dyadic rectangles.
  Let $E_{I\times J} = \scr E_{I\times J}^*$.
  Assume that $\{ E_{I\times J}\, :\, I, J \in \dint\} $ is a bi-tree, then the following hold
  \begin{enumerate}[(i)]
  \item The block basis $\{b_{I\times J } : I\times J \in \drec \}$ is equivalent to the
    2D-Haar basis in $\hardys$ with $C_1 = C_1(C_2)$.
  \item If there exists $C_3 > 0 $ so that for each $I\times J, I_0\times J_0 \in \drec$ with
    $I\times J \supset I_0\times J_0$ and for every $K \times L \in \scr E_{I\times J}$ we have
    \begin{equation}\label{prelims:2d-jones}
      C_3^{-1} \frac{ |E_{I_0\times J_0}|  }{|E_{I\times J}|}
      \leq \frac{|( K\times L) \isect E_{I_0\times J_0}|}{| K\times L|} 
      \leq  C_3\, \frac{|E_{I_0\times J_0}|}{|E_{I\times J}|},
    \end{equation}
    then the orthogonal projection
    \begin{equation*}
      Qf = \sum \langle f ,b_{I\times J}\rangle \frac{b_{I\times J}}{\|b_{I\times J}\|_2^{2}}
    \end{equation*}
    defines a bounded operator on $\hardys$ with norm only depending on $C_3$ and $C_2$.
  \end{enumerate}
\end{thm}

\subsection*{Rademacher type functions in $\hardys$ and $\bmos$}\hfill\\

\noindent
We define the following Rademacher type system as block basis of the Haar system.
Given $r \geq k_0$ and $K_0\times L_0 \in \drec$ with $|K_0| = 2^{-k_0}$ we specify the following
functions.
First, for any choice of signs we set
\begin{equation*}
  d_i = \sum_{K \in D_i\isect K_0} \pm h_K,
  \qquad i \geq r.
\end{equation*}
Then it is easy to see that if we define
\begin{equation*}
  g_i(s,t) = d_i(s)\, h_{L_0}(t),
  \qquad s,t \in [0,1]
\end{equation*}
for each dyadic interval $L_0$, then by~\eqref{eq:definition_hardys_norm} and duality we have
\begin{equation}\label{eq:sum_of_rademacher-estimate}
  \big\| \sum_{i=r}^{r+k-1} g_i \big\|_{\hardys} = \sqrt{k} |L_0|
  \qquad\text{and}\qquad
  \big\| \sum_{i=r}^{r+k-1} g_i \big\|_{\bmo} = \sqrt{k}.
\end{equation}
%%%%%%%%%%%%%%%% END PRELIMINARIES %%%%%%%%%%%%%%%%

%%%%%%%%%%%%%%%% BEGIN RESULTS %%%%%%%%%%%%%%%%
\section{Localized facorization}\label{s:results}

Here we prove our quantitative factorization theorem which is one of the three major steps towards
the proof of our main theorem.

The main result of this paper is the following quantitative factorization
theorem~\ref{thm:local-hardy}.
\begin{thm}\label{thm:local-hardy}
  For $n \in \bb N$ and $M > 0$ there exists $N=N(n,M)$ so that the following holds:
  For any operator $T\, :\, \hardys[N]\rightarrow \hardys[N]$ with $\|T\|\leq M$ the identity on
  $\hardys[n]$ factors through $H = T$ or $H = \Id-T$ such that
  \begin{equation*}
    \vcxymatrix{\hardys[n] \ar[r]^\Id \ar[d]_E & \hardys[n]\\
      \hardys[N] \ar[r]_H & \hardys[N] \ar[u]_P}
    \qquad \|E\|\|P\|\leq C,
  \end{equation*}
  where $C > 0$ is a universal constant.
\end{thm}

The proof is based on the following three theorems.
\begin{enumerate}[(i)]
\item Ramsey's theorem~\ref{thm:ramsey} for colored dyadic rectangles.
\item The projection theorem~\ref{pro:projection}.
\item The almost-diagonalization theorem~\ref{thm:quasi-diag} stated below.
\end{enumerate}
These three theorems combined provide the reduction from general operators in
theorem~\ref{thm:local-hardy} to multipliers on the Haar system.

\subsection*{The almost-diagonalization theorem}\hfill\\

\noindent
We now state the almost-diagonalization theorem~\ref{thm:quasi-diag} and show that in combination
with Ramsey's theorem~\ref{thm:ramsey} for colored dyadic rectangles and the projection
theorem~\ref{pro:projection} it yields the proof of our main result, Theorem~\ref{thm:local-hardy}.
\begin{thm}\label{thm:quasi-diag}
  Let $n \in \bb N$, $M > 0$ and
  $\{ \varepsilon_{I\times J}\, :\, I\times J \in \scr R_n \}$ be a given set of small positive
  scalars.
  Then there exists $N = N(n,M, \{\varepsilon_{I\times J}\})$ such that for any linear operator
  $T : \hardys[N]\rightarrow \hardys[N]$ with $\|T\| \leq M$ there exist disjoint collections
  $\scr E_{I\times J}$, indexed by $I\times J \in \drec_n$, consisting of pairwise disjoint dyadic
  rectangles defining the functions
  \begin{equation*}
    b_{I\times J} = \sum_{K\times L \in \scr E_{I\times J}} h_{K\times L},
  \end{equation*}
  which satisfy the following conditions:
  
  \begin{enumerate}[(i)]
  \item $\scr E_{I\times J} \subset \drec_N$ and $|b_{I\times J}| \leq 1$
    for all $I\times J \in \drec_n$.
    \label{enu:thm:quasi-diag-1}

  \item The orthogonal projection
    \begin{equation*}
      Q(f) = \sum_{I\times J \in \drec_n}
      \big\langle f, \frac{b_{I\times J}}{\|b_{I\times J}\|_2} \big\rangle\,
      \frac{b_{I\times J}}{\|b_{I\times J}\|_2}
    \end{equation*}
    is a bounded operator on $\hardys$ with $Q(\hardys) = \spn\{b_{I\times J}\}$ satisfying
    \begin{equation*}
      \|Q\, :\, \hardys \to \hardys\| \leq C_2,
    \end{equation*}
    for some universal constant $C_2 > 0$.
    \label{enu:thm:quasi-diag-2}

  \item The map $S\, :\, \hardys[n] \rightarrow \spn\{b_{I\times J}\} \isect \hardys[N]$ defined as
    the linear extension of $h_{I\times J} \mapsto b_{I\times J}$ is an isomorphism with
    \begin{equation}\label{thm:quasi-diag-ii}
      \|S\| \|S^{-1}\| \leq C_3,
    \end{equation}
    for some universal constant $C_3 > 0$.
    \label{enu:thm:quasi-diag-3}

  \item We have the estimate
    \begin{equation}\label{thm:quasi-diag-iii}
      \sum_{K\times L \neq I\times J} | \langle T b_{K\times L}, b_{I\times J} \rangle |
      \leq \varepsilon_{I \times J} \|b_{I\times J}\|_2^2,
    \end{equation}
    for all $I \times J \in \drec_n$.
    \label{enu:thm:quasi-diag-4}
  \end{enumerate}
\end{thm}

\noindent
The proof of the almost-diagonalization theorem~\ref{thm:quasi-diag} is given in Section~\ref{s:reduction}.

\subsection*{Proof of Theorem~\ref{thm:local-hardy}}\hfill\\

\noindent
Let $n \in \bb N$, $M > 0$.
We define $N$ by the chain of the following conditions:
\begin{equation}\label{eq:definition_dimensions}
  \begin{aligned}
    N & = N(N_1,M,\{\varepsilon_{I\times J}\}),
    &  N_1 &= N_2\, 2^{4^{N_2}},
    & N_2 &= n4^n,
  \end{aligned}
\end{equation}
where $\{\varepsilon_{I\times J}\, :\, I\times J \in \drec_{N_1}\}$ is a collection of positive
scalars satisfying
\begin{equation}\label{eq:definition_epsilon}
  \sum_{I\times J \in \drec_{N_1}} \varepsilon_{I\times J}
  \leq \frac{1}{4}.
\end{equation}
Let $T\, :\, \hardys[N]\rightarrow \hardys[N]$ be an operator such that $\|T\| \leq M$.
Now apply Theorem~\ref{thm:quasi-diag} to $T$.
This gives a block basis $\{b_{I\times J}\, :\, I\times J \in \drec_{N_1}\}$ satisfying the
conclusions~\eqref{enu:thm:quasi-diag-1}--\eqref{enu:thm:quasi-diag-4} of
Theorem~\ref{thm:quasi-diag}.
The Ramsey theorem~\ref{thm:ramsey} for colored dyadic rectangles applied to
\begin{equation*}
  \scr C = \{ I\times J \in \drec_{N_1}\, :\,
    |\langle T b_{I\times J}, b_{I\times J} \rangle| \geq \|b_{I\times J}\|_2^2/2
  \}
\end{equation*}
yields collections $\scr A, \scr B \subset \dint^{N_1}$, with Carleson constants
$\cc{\scr A} \geq N_2$ and $\cc{\scr B} \geq N_2$, such that
$\scr A \times \scr B \subset \scr C$ or
$\scr A \times \scr B \subset \drec_{N_1} \setminus \scr C$.
We choose $H = T$ if $\scr A \times \scr B \subset \scr C$ and
$H = \Id - T$ if $\scr A \times \scr B \subset \drec_{N_1} \setminus \scr C$.

The following lower estimate will be essential below:
\begin{equation}\label{eq:H-large}
  |\langle H b_{I\times J}, b_{I\times J} \rangle| \geq \|b_{I\times J}\|_2^2/2,
  \qquad I\times J \in \scr A\times \scr B.
\end{equation}
We define the product space $X_{\scr A\times \scr B}$ by
\begin{equation*}
  X_{\scr A\times \scr B}
  = \spn\{ h_{I\times J}\, :\, I\times J \in \scr A\times \scr B \} \isect \hardys.
\end{equation*}
Since $\cc{\scr A} \geq N_2$, $\cc{\scr B} \geq N_2$, we know
from~\eqref{prelims:eq:condensation-2d} that there exists a universal constant $C > 0$ so that
\begin{equation*}
  \vcxymatrix{
    \hardys[n] \ar[r]^\Id \ar[d]_{E_0} & \hardys[n]\\
    X_{\scr A\times \scr B} \ar[r]^\Id & X_{\scr A\times \scr B} \ar[u]_{P_0}
  }
  \qquad \|E_0\|\|P_0\| \leq C.
\end{equation*}
We claim that Theorem~\ref{thm:quasi-diag} and the choices we made
in~\eqref{eq:definition_dimensions},\eqref{eq:definition_epsilon} and~\eqref{eq:H-large} imply that
there exist linear operators $S_1$ and $P_1$ such that
\begin{equation*}
  \vcxymatrix{X_{\scr A\times \scr B} \ar[r]^\Id \ar[d]_{S_1} & X_{\scr A\times \scr B}\\
    \hardys[N] \ar[r]^H & \hardys[N] \ar[u]_{P_1}
  }
  \qquad \|E_1\|\|P_1\| \leq C,
\end{equation*}
for some universal constant $C > 0$.
For the verification of the claim we remark that the method lined out
in~\cite[288--290]{mueller:2005} is directly applicable:
The isomorphic embedding
\begin{equation*}
  S_1\, :\, X_{\scr A\times \scr B} \to \spn\{b_{I\times J}\, :\, I\times J \in \scr A\times \scr B\}
\end{equation*}
is defined as the linear extension of the map
\begin{equation*}
  h_{I\times J} \mapsto b_{I\times J}.
\end{equation*}
For the norm estimate of $S_1$ we refer to~\eqref{thm:quasi-diag-ii}.
Next, define
\begin{equation*}
  \widetilde P_1\, :\, \hardys[N]\to \spn\{b_{I\times J}\, :\, I\times J \in \scr A\times \scr B\}
\end{equation*}
by the formula
\begin{equation*}
  f\mapsto \sum_{I\times J \in \scr A\times \scr B} \langle f, b_{I\times J}\rangle b_{I\times J}
  \langle Hb_{I\times J}, b_{I\times J}\rangle^{-1},
\end{equation*}
and observe that $\|\widetilde P_1\| \leq 2 \|Q\|$.
We observe that for $g \in \spn\{b_{I\times J}\, :\, I\times J \in \scr A\times \scr B\}$ we have
\begin{equation*}
  \widetilde P_1Hg = g + Gg,
\end{equation*}
where the error term $Gg$ is controlled via
$2\sum_{I\times J \in \drec_{N_1}} \varepsilon_{I\times J} \leq 1/2$ by the following operator
norm estimate
\begin{equation*}
  \big\|G\, :\, \spn\{b_{I\times J}\, :\, I\times J \in \scr A\times \scr B\}
  \to \spn\{b_{I\times J}\, :\, I\times J \in \scr A\times \scr B\}\big\|_{\hardys} \leq \frac{1}{2}.
\end{equation*}
Hence, we may invert $\Id+G$ on $\spn\{b_{I\times J}\, :\, I\times J \in \scr A\times \scr B\}$ so that
\begin{equation*}
  (\Id+G)^{-1} \widetilde P_1Hg = g,
  \qquad g \in \spn\{b_{I\times J}\, :\, I\times J \in \scr A\times \scr B\}.
\end{equation*}
Note that $\|(\Id+G)^{-1}\| \leq 2$.
This defines $P_1$ as follows:
\begin{equation*}
  P_1 f = S_1^{-1}(\Id+G)^{-1} \widetilde P_1 f,
  \qquad f \in X_{\scr A\times \scr B}.
\end{equation*}
We should emphasize that $S_1^{-1}$ is well defined on the range of $(\Id +G)^{-1}$ and furthermore
$(\Id +G)^{-1}$ is well defined on the range of $\widetilde P_1$.

Finally, it remains to merge the diagrams yielding the following factorization:
\begin{equation*}
  \vcxymatrix{\hardys[n] \ar[r]^\Id \ar[d]_{E} & \hardys[n]\\
    \hardys[N] \ar[r]^H & \hardys[N] \ar[u]_{P}
  }
  \qquad \|E\|\|P\| \leq C,
\end{equation*}
where $C > 0$ is a universal constant.
%%%%%%%%%%%%%%%% END RESULTS %%%%%%%%%%%%%%%%

%%%%%%%%%%%%%%%% BEGIN LOCAL %%%%%%%%%%%%%%%%
\section{Quantitative almost-diagonalization}\label{s:reduction}

\noindent
In this section we give the proof of Theorem~\ref{thm:quasi-diag}.
Our argument is inductive.
We use induction within the collection of dyadic rectangles.
It is therefore important that we introduce a suitable linear ordering relation on the collection of
dyadic rectangles.
Below we specifically construct the linear ordering relation $\drless$ so that the bijective index
function $\drindex\, :\, \drec \rightarrow \bb N$, which is defined by
\begin{equation*}
  \drindex(R_0) < \drindex(R_1)
  \Leftrightarrow R_0 \drless R_1,
  \qquad R_0,R_1 \in \drec,
\end{equation*}
has the following properties~\eqref{eq:ordering-1}
and~\eqref{eq:ordering-2}.
For a picture of the index function $\drindex$ see Figure~\ref{fig:ordering-relation}.
The geometry of a dyadic rectangle and its position within our linear ordering $\drless$ are linked
by the inequalities
\begin{equation}\label{eq:ordering-1}
  (2^k -1)^2 < \drindex(I \times J) \leq (2^{k+1} -1)^2,
  \qquad \text{whenever $\min(|I|,|J|) = 2^{-k}$},
\end{equation}
as well as
\begin{equation}\label{eq:ordering-2}
  4\, | I_1\times J_1 | \leq \frac{|I_0\times J_0|}{\min(|I_1|,|J_1|)},
  \qquad \text{whenever $I_0\times J_0 \drless I_1\times J_1$}.
\end{equation}
Any linear orderings on the dyadic rectangles for which~\eqref{eq:ordering-1}
and~\eqref{eq:ordering-2} hold may serve as basis for our induction argument in the proof of
Theorem~\ref{thm:quasi-diag}.

\subsection{Constructing the linear ordering relation $\drless$ on $\drec$}\label{ss:linear_order}\hfill\\

\noindent
First, we define the rectangles of fixed side lengths $2^{-m}$ and $2^{-n}$ by setting
\begin{equation}\label{eq:linear_order-macroblock}
  \scr B_{m,n} = \{I\times J \in \drec\, :\, |I|=2^{-m}, |J|=2^{-n}\},
  \qquad m,n \geq 0.
\end{equation}
Second, we will define the ordering relation $\dilesslex$ on each of the blocks $\scr B_{m,n}$.
Given two dyadic rectangles $I_0\times J_0,\, I_1\times J_1 \in \scr B_{m,n}$ we set
\begin{equation*}
  I_0\times J_0 \dilesslex I_1\times J_1\,
  :\Leftrightarrow (\inf I_0,\inf J_0) \lesslex (\inf I_1,\inf J_1),
\end{equation*}
where $\lesslex$ denotes the lexicographic ordering on $\bb R^2$.
Third, we shall collect the blocks $\scr B_{m,n}$ in the collections
\begin{equation*}
  \scr S_k = \{\scr B_{m,n}\, :\, \max(m,n) = k\},
  \qquad k \geq 0.
\end{equation*}
Third, we need to bring the blocks $\scr B_{m,n}$ in order.
To this end, we consider
\begin{equation*}
  w : \{\scr B_{m,n} : m,n \geq 0\} \rightarrow \bb N_0
\end{equation*}
such that the following conditions hold for all $k \geq 1$:
\begin{enumerate}[(i)]
\item $w|\scr S_k\, :\, \scr S_k \rightarrow \{k^2,\ldots,(k+1)^2-1\}$ is bijective.
\item we set $w(\scr B_{0,k}) = k^2$ and moreover
  \begin{equation*}
    w(\scr B_{m_0,n_0}) < w(\scr B_{m_1,n_1}) \Leftrightarrow
    \begin{cases}
      m_0 > n_0 \text{ and } m_1 \leq n_1,\\
      m_0 > n_0 \text{ and } m_1 > n_1 \text{ and } n_0 < n_1,\\
      m_0 \leq n_0 \text{ and } m_1 \leq n_1 \text{ and } m_0 < m_1,
    \end{cases}
  \end{equation*}
  for all $\scr B_{m_0,n_0}, \scr B_{m_1,n_1} \in \scr S_k\setminus\{\scr B_{0,k}\}$.
\end{enumerate}
Finally, we use the function $w$ and its properties as well as the properties of $\dilesslex$ to
define our linear ordering relation $\drless$ on the dyadic rectangles $\drec$.
If $I_0 \times J_0,\, I_1 \times J_1 \in \drec$ we set
\begin{equation*}
  (I_0 \times J_0) \drless (I_1 \times J_1) :\Leftrightarrow
  \begin{cases}
    w(\scr B_{\lev I_0,\lev J_0}) < w(\scr B_{\lev I_1,\lev J_1}) \text{ or}\\
    w(\scr B_{\lev I_0,\lev J_0}) = w(\scr B_{\lev I_1,\lev J_1}) \text{ and } (I_0,J_0)\dilesslex (I_1,J_1).
  \end{cases}
\end{equation*}
Since our ordering relation $\drless$ is linear, we may well define the bijective index function
$\drindex\, :\, \drec \rightarrow \bb N$ by the following property:
\begin{equation*}
  \drindex(R_0) < \drindex(R_1)
  \Leftrightarrow R_0 \drless R_1,
  \qquad R_0,R_1 \in \drec.
\end{equation*}
Observe that the crucial relations between the geometry of a dyadic rectangle and its position
within our linear ordering~\eqref{eq:ordering-1} and~\eqref{eq:ordering-2} are satisfied by design.
\begin{figure}[bt]
  \begin{center}
    % \begin{asy}%[height=5cm,width=\the\linewidth,inline=true]
    %   unitsize(2cm);
    %   defaultpen(fontsize(10pt));

    %   access "images/haar.asy" as haar;
      
    %   //define which numbers will be displayed (default = not displayed)
    %   int maxnumber = 49;
    %   for(int cnt = 0; cnt < maxnumber; ++cnt) {
    %     haar.numberingselector.push(true);
    %   }

    %   haar.drawmacroblocks(3,.05);
    % \end{asy}
    \includegraphics{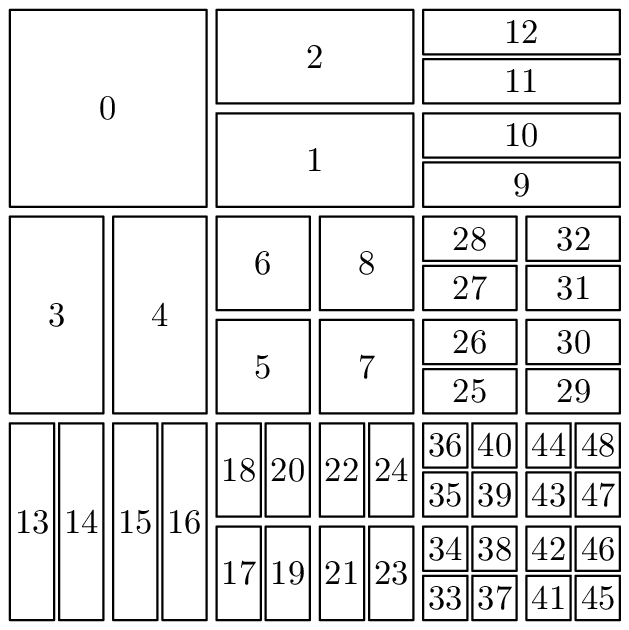}
  \end{center}
  \caption{Index function $\drindex(I\times J)$ for all $I\times J \in \drec_2$.}
  \label{fig:ordering-relation}
\end{figure}

\subsection{Combinatorial lemma}\label{ss:combinatorial}\hfill\\

\noindent
Let $\{r_i\}$ denote the sequence of independent Rademacher functions which are given by
\begin{equation*}
  r_i(t) = \sign(\sin(2\pi 2^i t)),
  \qquad t \in [0,1],\, i \in \bb N.
\end{equation*}
We consider the tensor product $r_{i,j}$ of the standard Rademacher system defined as
\begin{equation*}
  r_{ij}(s,t) = r_i(s)\, r_j(t),
  \qquad (s,t) \in [0,1]^2
\end{equation*}
It is well known and easy to verify that in both spaces, $\hardys$ and $\bmos$, the system
$\{r_{ij}\}$ is equivalent to the unit vector basis of $\ell^2$.
Specifically, there exists constants $c_0$, $C_0$ so that for any sequence of scalars
$\{a_{ij}\}$ the following inequalities hold.
\begin{equation*}
  \| \sum a_{ij} r_{ij} \|_{\hardys}^2 = \sum a_{ij}^2
\end{equation*}
and
\begin{equation*}
  c_0 \sum a_{ij}^2  \leq \| \sum a_{ij} r_{ij} \|_{\bmos}^2 \leq  C_0 \sum a_{ij}^2.
\end{equation*}
Hence, $\{r_{ij}\}$ is a weak null sequence in both spaces $\hardys$ and $\bmos$,
\begin{align*}
  r_{ij} \rightarrow 0
  \qquad &\text{weakly in $\hardys$, if $i \rightarrow \infty$ or $j \rightarrow \infty$}\\
  \intertext{and}
  r_{ij} \rightarrow 0
  \qquad &\text{weakly in $\bmos$, if $i \rightarrow \infty$ or $j \rightarrow \infty$}.
\end{align*}
For the purpose of our present work we need a quantitative strengthening of these considerations.
This is done in the following combinatorial lemma.
Our combinatorial argument is controlled by the local frequency weight
\begin{equation*}
  f(K\times L) = |\langle x, h_{K\times L} \rangle| + |\langle y, h_{K\times L} \rangle|,
  \qquad K\times L \subset K_0\times L_0
\end{equation*}
where $x \in \bmos$ and $y\in \hardys$ are fixed functions and $K_0\times L_0 \in \drec$.
For us, it will be extremely important that the collection
\begin{equation*}
  \{K\times L\, :\, f(K\times L) \leq \tau |K\times L|\}
\end{equation*}
contains almost complete and well--structured coverings of $K_0\times L_0$ of the form
\begin{equation*}
  \{K_0 \times L\, :\, L \in \dint_{\ell} \isect L_0\}
  \qquad\text{and}\qquad
  \{K \times L_0\, :\, K \in \dint_k \isect K_0\},
\end{equation*}
with $k$ and $\ell$ well under control in terms of $\tau$.
See Figure~\ref{fig:combinatorial-2}.

\begin{lem}\label{lem:comb-1}
  Let $i \in \bb N$, $K_0, L_0 \in \dint$, $x_j \in \bmos$, $y_j \in \hardys$, $1 \leq j \leq i$,
  such that
  \begin{equation}\label{eq:function-hypothesis}
    \sum_{j=1}^i \|x_j\|_{\bmos} \leq 1
    \qquad\text{and}\qquad
    \sum_{j=1}^i \|y_j\|_{\hardys} \leq |K_0\times L_0|.
  \end{equation}
  Let $\tau > 0$, $r \in \bb N_0$, $K\times L\in \drec$ and define the local frequency weight
  \begin{equation}\label{eq:local_frequency_weight}
    f_i(K\times L)
    = \sum_{j=1}^i |\langle x_j, h_{K\times L} \rangle| + |\langle y_j, h_{K\times L}\rangle|
  \end{equation}
  as well as the collections
  \begin{align*}
    \scr K(K_0\times L_0) &= \big\{ K\times L_0\, :\, K \subset K_0,\ |K|\leq 2^{-r}|K_0|,\
      f_i(K\times L_0)
      \leq \tau |K\times L_0|
    \big\},\\
    \scr L(K_0\times L_0) &= \big\{ K_0\times L\, :\, L \subset L_0,\ |L|\leq 2^{-r}|L_0|,\
      f_i(K_0\times L)
      \leq \tau |K_0\times L|
    \big\}.
  \end{align*}
  For all integers $k,\ell$ the collections $\scr K_k(K_0\times L_0)$ and
  $\scr L_\ell(K_0\times L_0)$ are given by
  \begin{align*}
    \scr K_k(K_0\times L_0)
    &= \scr K(K_0\times L_0) \isect (\{K \in \dint\, :\, |K| = 2^{-k}|K_0|\}\times \dint),\\
    \scr L_\ell(K_0\times L_0)
    & = \scr L(K_0\times L_0) \isect (\dint\times \{L \in \dint\, :\, |L| = 2^{-\ell}|L_0|\}).
  \end{align*}
  Let $\delta > 0$.
  Then there exist integers $k,\ell$ with
  \begin{equation}\label{lem:comb-1:int}
    r \leq k,\ell \leq \lfloor \frac{i^2}{\delta^2\tau^2} \rfloor + r
  \end{equation}
  such that
  \begin{equation}\label{lem:comb-1:measure}
    | \scr K_k^*(K_0\times L_0) | \geq (1-\delta) |K_0\times L_0|
    \quad\text{and}\quad
    | \scr L_\ell^*(K_0\times L_0) | \geq (1-\delta) |K_0\times L_0|.
  \end{equation}
\end{lem}

\begin{proof}
  Define $\scr B = \{K\times L_0\, :\, K\subset K_0\} \setminus \scr K(K_0\times L_0)$ and
  \begin{equation*}
    \scr B_k = \scr B \isect (\{K \in \dint\, :\, |K| = 2^{-k}|K_0|\}\times \dint),
  \end{equation*}
  see Figure~\ref{fig:combinatorial-1}.
  \begin{figure}[bt]
    \begin{center}
      % \begin{asy}%[height=5cm,width=\the\linewidth,inline=true]
      %   unitsize(1cm);
      %   defaultpen(fontsize(10pt));

      %   // $\scr K_k(K_0\times L_0)$
      %   int width = 10;
      %   int height = 3;
      %   int[] selector = {
      %     1,1,0,1,
      %     0,1,1,0,
      %     1,1,1,0,
      %     0,0,1,1,
      %     0,1,1,1,
      %     1,1,0,0,
      %     1,1,0,1,
      %     0,1,0,1
      %   };

      %   //big rectangle
      %   draw(box((0,0),(width,height)),linewidth(1));
      %   label("$K_0$",(width/2,height),N);
      %   label("$L_0$",(0,height/2),W);

      %   real intervallength = width/32;//interval length
      %   for(int i=0; i < selector.length; ++i){
      %     transform t = shift(i*(intervallength,0));
      %     //small gray rectangles + intervals
      %     if(selector[i]==1){
      %       pair x0 = t*(0,-.25);//begin point for interval
      %       pair x1 = t*(intervallength,-.25);//end point for interval
      %       draw(x0--x1,Bars(2));//interval
      %       filldraw(box((x0.x,0),(x1.x,height)),mediumgray);//small rectangles
      %     }//if
      %     //small white rectangles + intervals
      %     if(selector[i]==0){
      %       pair x0 = t*(0,-.75);//begin point for interval
      %       pair x1 = t*(intervallength,-.75);//end point for interval
      %       draw(x0--x1,Bars(2));//interval
      %       draw(box((x0.x,0),(x1.x,height)));//small white rectangles
      %     }//if
      %   }//for
      % \end{asy}
      \includegraphics{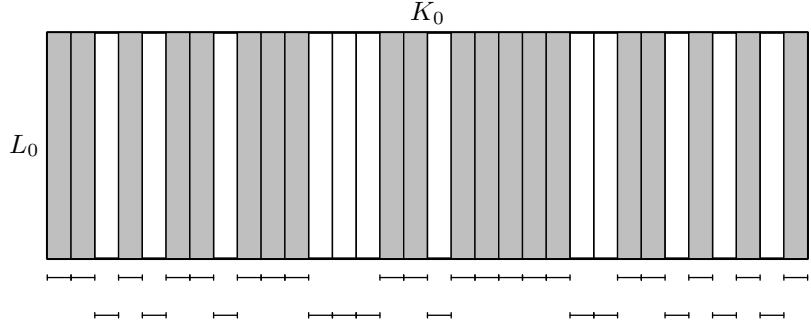}
    \end{center}
    \caption{The shaded rectangles form $\scr K_k(K_0\times L_0)$, the white rectangles form
      $\scr B_k$.
      The $x$--component of $\scr K_k(K_0\times L_0)$ (first level of intervals) and the
      $x$--component of $\scr B_k$ (second level of intervals) form a disjoint cover  of $K_0$.}
    \label{fig:combinatorial-1}
  \end{figure}
  Let
  \begin{equation*}
    A = \lfloor \frac{i^2}{\delta^2\tau^2} \rfloor + r.
  \end{equation*}
  By construction $\scr B_k$ and $\scr K_k(K_0\times L_0)$ form a disjoint decomposition of
  $K_0\times L_0$.
  We will determine a collection $\scr K_k(K_0\times L_0)$ by showing that $\scr B_k^*$ is small
  enough for at least one value of $k$.
  Now assume the opposite, namely that
  \begin{equation*}
    | \scr B_k^* | \geq \delta |K_0\times L_0|,
    \qquad r \leq k \leq A.
  \end{equation*}
  Summing these estimates yields
  \begin{equation}\label{lem:comb-1:1}
    \sum_{k=r}^A | \scr B_k^* | \geq (A-r+1)\, \delta\, |K_0\times L_0|,
  \end{equation}
  Observe that
  \begin{align*}
    \tau\cdot \sum_{k=r}^A |\scr B_k^*|
    &\leq \sum_{j=1}^i \sum_{k=r}^A \sum_{K\times L_0 \in \scr B_k}
      |\langle x_j, h_{K\times L_0} \rangle| + |\langle y_j, h_{K\times L_0}\rangle|\\
    &= \sum_{j=1}^i \big|\big\langle
        x_j, \sum_{k=r}^A \sum_{K\times L_0 \in \scr B_k} \pm h_{K\times L_0}
      \big\rangle\big|
      + \big|\big\langle
        y_j,\sum_{k=r}^A \sum_{K\times L_0 \in \scr B_k} \pm h_{K\times L_0}
      \big\rangle\big|.
  \end{align*}
  By~\eqref{eq:sum_of_rademacher-estimate} we have
  \begin{align*}
    \Big\| \sum_{k=r}^A \sum_{K\times L_0 \in \scr B_k} \pm h_{K\times L_0} \Big\|_{\hardys}
    & = \sqrt{A-r+1}\ |K_0\times L_0|,\\
    \Big\| \sum_{k=r}^A \sum_{K\times L_0 \in \scr B_k} \pm h_{K\times L_0} \Big\|_{\bmos}
    & = \sqrt{A-r+1},
  \end{align*}
  thus, by duality and~\eqref{eq:function-hypothesis} we obtain
  \begin{equation}\label{lem:comb-1:2}
    \tau\cdot \sum_{k=r}^A |\scr B_k^*|
    \leq i\, \sqrt{A-r+1}\ |K_0\times L_0|.
  \end{equation}
  Combining~\eqref{lem:comb-1:1} and~\eqref{lem:comb-1:2} we conclude
  \begin{equation*}
    A \leq \frac{i^2}{\delta^2 \tau^2} + r -1,
  \end{equation*}
  which contradicts the definition of $A$.
  Thus we found $r\leq k\leq A$ so that
  \begin{equation*}
    | \scr K_k^*(K_0\times L_0) | \geq (1-\delta) |K_0\times L_0|,
  \end{equation*}
  see Figure~\ref{fig:combinatorial-1}.
  We emphasize that the $x$--component of the rectangles in $\scr K_k^*(K_0\times L_0)$ are covering
  a set of measure $\geq (1-\delta) |K_0|$ in $K_0$.
  The $y$--component of the rectangles in $\scr K_k^*(K_0\times L_0)$ equals $L_0$ throughout.
  (For the later use of this lemma it is extremely important that we found a large collection of
  rectangles $\scr K_k(K_0\times L_0)$ where the $y$--component $L_0$ remains intact.)
  See Figure~\ref{fig:combinatorial-2}.
  \begin{figure}[bt]
    \begin{center}
      \includegraphics{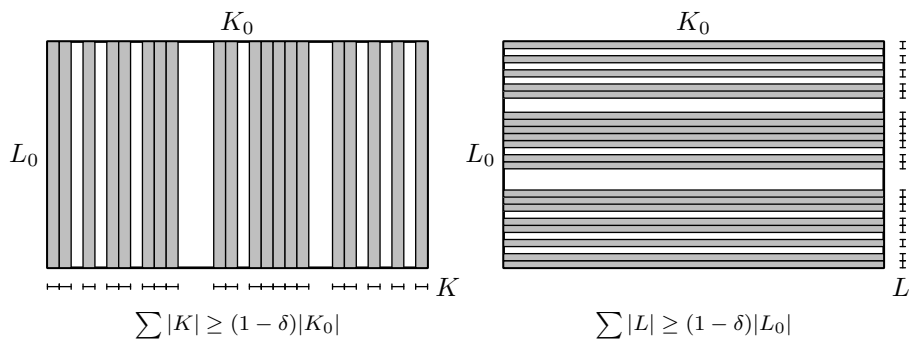}
    \end{center}
    \caption{Two good covers.
      The shaded rectangles in the left picture are from $\scr K_k(K_0\times L_0)$.
      For those rectangles the $y$--component $L_0$ remains intact and the $x$--components $K$ form a
      large cover of $K_0$.
      \textbf{Therefore} the construction in Section~\ref{ss:quasi-diag} yields the crucial measure
      estimate~\eqref{eq:thm:quasi-diag:13}.
      The right picture displays the collection $\scr L_\ell(K_0\times L_0)$ and the roles of $x$ and
      $y$--components are interchanged.}
    \label{fig:combinatorial-2}
  \end{figure}

  The same proof in the other variable can be used to show the estimate for $\scr L_\ell^*$.
\end{proof}

\begin{figure}[bt]
  \begin{center}
    \includegraphics{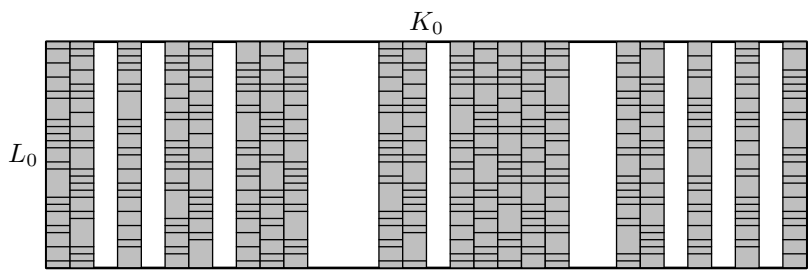}
\end{center}
\caption{A bad cover of $K_0\times L_0$.
  These fragmented shaded rectangles cover the same subset of $K_0\times L_0$ as
  $\scr K_k(K_0\times L_0)$ (see Figure~\ref{fig:combinatorial-2}).
  The $y$--component $L_0$ did not remain intact and \textbf{therefore} the construction in
  Section~\ref{ss:quasi-diag} would not yields the crucial measure
  estimate~\eqref{eq:thm:quasi-diag:13}.}
\label{fig:combinatorial-3}
\end{figure}

\subsection{Proof of Theorem~\ref{thm:quasi-diag}}\label{ss:quasi-diag}\hfill\\

\noindent
Theorem~\ref{thm:quasi-diag} asserts that we are able to construct a large block basis
$\{b_{I\times J}\}$ in $\hardys[N]$ which are almost eigenvectors for $T$.
Moreover, the block basis is such that it spans a well complemented copy of $\hardys[n]$ in
$\hardys[N]$.
We choose the normalization $M = 1$ and $\|T\|\leq 1$.

It is here where we will exploit our linear order $\drless$ introduced on the collection of dyadic
rectangles $\drec$.
The proof described below is by mathematical induction executed along the linear order given
by $\drindex$.

\subsection*{Inductive construction}\hfill\\

\noindent
To make the transition from standard indexing by dyadic rectangles to indexing by natural numbers we
employ the following convention.
Given a dyadic rectangle $I\times J$ with $\drindex(I \times J) = i$ we will systematically relabel
the collections $\scr E_{I\times J}$, the functions $b_{I\times J}$ and the constants
$\delta_{I\times J}$, $\tau_{I\times J}$ by $\scr E_i$, $b_i$ and $\delta_i$, $\tau_i$,
respectively.

Before we begin with our construction we explicitly define the constants $\delta_i$ by
\begin{equation}\label{eq:constant_definition}
  \delta_i = 2^{-i}/(8 n).
\end{equation}
The remaining crucial constants $\tau_i$ will be defined inductively as the construction proceeds.

\subsubsection*{First stage of the induction}
We begin the induction by setting $\scr E_1 := \scr E_{[0,1]\times [0,1]} :=\{[0,1]\times [0,1]\}$
and $b_1 := b_{[0,1]\times [0,1]} := h_{[0,1]\times [0,1]}$.

\subsubsection*{At stage $i$ of the induction}
We assume that we have already defined the disjoint collections of dyadic rectangles $\scr E_j$ for
all $1 \leq j \leq i-1$.
Now, we will construct $\scr E_i$.
The construction of $\scr E_i$ depends crucially on the value of $i$.
We will distinguish between two principal cases, where the second one is divided again into two
sub cases.
\begin{enumerate}[$\triangleright$]
\item Case~1:
  The stage ordinal $i$ is given by $i = \drindex([0,1]\times J)$.
\item Case~2:
  The stage ordinal $i$ is given by $i = \drindex(I\times J)$, where $I\neq [0,1]$.
  \begin{enumerate}[$+$]
    \item Case~2.a: The second component $J$ satisfies $J = [0,1]$.
    \item Case~2.b: The second component $J$ satisfies $J \neq [0,1]$.
  \end{enumerate}
\end{enumerate}

%%%%%%%%%%%%%%%% BEGIN STEP 1 %%%%%%%%%%%%%%%%
\noindent
\begin{minipage}[H]{.7\textwidth}
  \textbf{Case~1: $\bm{I = [0,1]}$.}
  The stage ordinal $i$ is given by $i = \drindex([0,1] \times J)$.
  Case~1 is applicable to the light rectangles.
  The collections $\scr E_{I_0\times J_0}$ indexed by the dark rectangles $I_0\times J_0$ are
  already well defined at this stage.
  The white ones are ignored.
\end{minipage}
\begin{minipage}[H]{.3\textwidth}
  \begin{center}
    % \begin{asy}
    %   unitsize(.75cm);
    %   defaultpen(fontsize(10pt));

    %   access "images/haar.asy" as haar;

    %   //define fillcolors (no fillcolor == invisible)
    %   int color_1 = 9;//first color_1 rectangles will be dark
    %   int color_2 = 13 ;//then light up to color_2, then fillcolor==invisible
    %   for(int cnt = 0; cnt < color_1; ++cnt){
    %     haar.fillcolor.push(heavygray);
    %   }
    %   for(int cnt = color_1; cnt < color_2; ++cnt) {
    %     haar.fillcolor.push(mediumgray);
    %   }
      
    %   //define which numbers will be displayed (default = not displayed)
    %   int maxnumber = 0;
    %   for(int cnt = 0; cnt < maxnumber; ++cnt) {
    %     haar.numberingselector.push(true);
    %   }

    %   haar.drawmacroblocks(3,.1);
    % \end{asy}
    \includegraphics{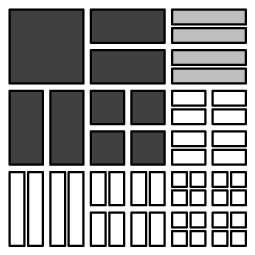}
  \end{center}
\end{minipage}
Recall that
\begin{equation*}
  b_j = \sum_{K\times L \in \scr E_j} h_{K\times L},
  \qquad 1 \leq j \leq i-1.
\end{equation*}
Since the collection $\scr E_j$ consists of pairwise disjoint rectangles
we have by~\eqref{eq:definition_hardys_norm} and duality that
\begin{equation*}
  \|b_j\|_{\bmos} = 1
  \qquad\text{and}\qquad
  \|b_j\|_{\hardys} = |\scr E_j^*|.
\end{equation*}
Let $\widetilde J$ denote the unique dyadic interval satisfying $\widetilde J \supset J$ and
$|\widetilde J| = 2|J|$.
By definition of our linear ordering we have $\drindex([0,1]\times \widetilde J) \leq i-1$.
Hence, $\scr E_{[0,1]\times \widetilde J}$ is already defined.
Now put
\begin{equation}\label{eq:proof:step-1:building_block_size}
  \beta_i = \min\{|K_0\times L_0|\, :\, K_0\times L_0\in \scr E_{[0,1]\times \widetilde J}\}
\end{equation}
and define for all $1 \leq j \leq i-1$
\begin{align}
  x_j &:= \frac{1}{i-1} T^*b_j,
  &y_j &:= \frac{\beta_i}{(i-1) |\scr E_j^*|} T b_j,
  \label{eq:proof:step-1:functions}
\end{align}
Recall that we are using the normalization $\|T\| \leq 1$, hence
\begin{equation*}
  \sum_{j=1}^{i-1} \|x_j\|_{\bmos} \leq 1
  \qquad\text{and}\qquad
  \sum_{j=1}^{i-1} \|y_j\|_{\hardys} \leq \beta_i.
\end{equation*}
We define the local frequency weight
\begin{equation}\label{eq:local_frequency_weight:1}
  f_{i-1}(K\times L)
  = \sum_{j=1}^{i-1} |\langle x_j, h_{K\times L} \rangle| + |\langle y_j, h_{K\times L}\rangle|,
  \qquad K\times L\in \drec.
\end{equation}
Given $L_0$ we remark that by our previous choices we have the following convenient implication:
\begin{equation}\label{eq:proof:step-1:first-coordinate}
  K\times L_0 \in \scr E_{[0,1]\times \widetilde J}
  \quad\text{implies}\quad K = [0,1].
\end{equation}
We now define the constant $\tau_i$ by
\begin{equation}\label{eq:proof:step-1:constant_definition}
  \tau_i
  = \frac{2^{-i}}{4(i-1)} \beta_i \min_{j\leq i} \varepsilon_j \|b_j\|_2^2
\end{equation}
For all $L_0$ such that $[0,1]\times L_0 \in \scr E_{[0,1]\times \widetilde J}$, we define the
collection of dyadic rectangles
\begin{equation*}
  \scr L([0,1]\times L_0) = \big\{ [0,1] \times L\, :\, L \subsetneq L_0,\
  f_{i-1}([0,1]\times L) \leq \tau_i |L|
  \big\}.
\end{equation*}
Applying Lemma~\ref{lem:comb-1} to $\scr L([0,1]\times L_0)$ yields an integer
$\ell = \ell([0,1]\times L_0)$ so that
\begin{equation}\label{eq:step-1:integer-bound}
  1 \leq \ell([0,1]\times L_0) < \frac{(i-1)^2}{\delta_i^2\, \tau_i^2} +1
\end{equation}
such that the collection of disjoint dyadic rectangles
\begin{equation*}
  \scr Z_{[0,1]\times J}([0,1]\times L_0)
  = \big\{ [0,1]\times L\in \scr L([0,1]\times L_0)\, :\, |L|=2^{-\ell([0,1]\times L_0)}|L_0| \big\}
\end{equation*}
satisfies the estimate
\begin{equation}\label{eq:step-1:local-measure-estimate}
  (1-\delta_i) |[0,1] \times L_0|
  \leq |\scr Z_{[0,1]\times J}^*([0,1]\times L_0)|
  \leq |[0,1] \times L_0|.
\end{equation}
Note that in Lemma~\ref{lem:comb-1} $\scr Z_{[0,1]\times J}([0,1]\times L_0)$ was denoted
$\scr L_\ell([0,1]\times L_0)$.
Now we take the union and define
\begin{equation*}
  \scr Z_{[0,1]\times J} = \Union\big\{ \scr Z_{[0,1]\times J}([0,1]\times L_0)\, :\,
  [0,1]\times L_0\in \scr E_{[0,1]\times \widetilde J}
  \big\}.
\end{equation*}
Since $\scr Z_{[0,1]\times J}([0,1]\times L_0) \subset \scr L([0,1]\times L_0)$, we know
\begin{equation}\label{eq:step-1:quasi-diag-estimate}
  f_{i-1}([0,1]\times L)
  \leq \tau_i |L|,
  \qquad \text{for $[0,1]\times L \in \scr Z_{[0,1]\times J}$}.
\end{equation}
\begin{figure}[bt]
  \begin{center}
    \includegraphics{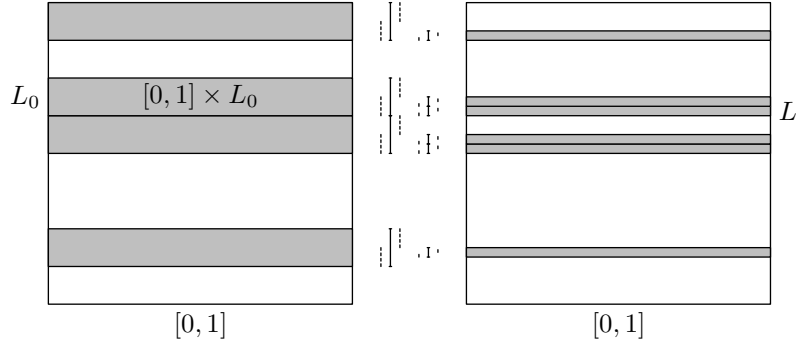}
  \end{center}
  \caption{This figure displays the transition from $\scr E_{[0,1]\times \widetilde J}$ to
    $\scr E_{[0,1]\times J}$ by means of a generic Gamlen-Gaudet
    step~\eqref{eq:step-1:block-basis-collection}.
    The shaded rectangles $[0,1]\times L_0$ on the left form $\scr E_{[0,1]\times \widetilde J}$,
    the shaded rectangles $[0,1]\times L$ on the right form $\scr E_{[0,1]\times J}$.
    The union of the rectangles in $\scr E_{[0,1]\times J}$ is contained in
    $E_{[0,1]\times \widetilde J}^\ell$, since in this figure $J$ is the left half of $\widetilde J$.
    The center displays the $y$--components of $b_{[0,1]\times \widetilde J}$ (center left) and
    of $b_{[0,1]\times J}$ (center right).
  }
  \label{fig:gg-1}
\end{figure}
We are now ready to define $\scr E_{[0,1]\times J}$ using the Gamlen-Gaudet procedure.
To this end recall first that $\widetilde J$ denotes the unique dyadic interval satisfying
$\widetilde J \supset J$ and $|\widetilde J| = 2|J|$.
For a dyadic interval $L_0$ we denote its left half by
$L_0^\ell$ ($L_0^\ell \subset L_0$, $|L_0^\ell| = |L_0|/2$, $\inf L_0^\ell = \inf L_0$)
and its right half by $L_0^r$ ($L_0^r = L_0 \setminus L_0^\ell$).
We define the sets
\begin{equation*}
  E_{[0,1]\times \widetilde J}^\ell
  = \Union_{[0,1]\times L_0 \in \scr E_{[0,1]\times \widetilde J}} [0,1]\times L_0^\ell
  \quad\text{and}\quad
  E_{[0,1]\times \widetilde J}^r
  = \Union_{[0,1]\times L_0 \in \scr E_{[0,1]\times \widetilde J}} [0,1]\times L_0^r.
\end{equation*}
If $J$ is the left half of $\widetilde J$, then we put
\begin{subequations}\label{eq:step-1:block-basis-collection}
  \begin{equation}\label{eq:step-1:block-basis-collection-a}
    \scr E_{[0,1] \times J} = \big\{ [0,1] \times L \in \scr Z_{[0,1]\times J}\, :\,
    [0,1] \times L\subset E_{[0,1]\times \widetilde J}^\ell
    \big\}.
  \end{equation}
  See Figure~\ref{fig:gg-1}.
  In the other case when $J$ is the right half of $\widetilde J$ we put accordingly
  \begin{equation}\label{eq:step-1:block-basis-collection-b}
    \scr E_{[0,1] \times J} = \big\{ [0,1] \times L \in \scr Z_{[0,1]\times J}\, :\,
    [0,1] \times L\subset E_{[0,1]\times \widetilde J}^r
    \big\}.
  \end{equation}
\end{subequations}

Recall that $i = \drindex([0,1]\times J)$ and $\delta_i = \delta_{[0,1]\times J}$.
An immediate consequence of the Gamlen-Gaudet construction
and~\eqref{eq:step-1:local-measure-estimate} is the estimate
\begin{equation}\label{eq:step-1:global-measure-estimate}
  \frac{1}{2} (1 - \delta_{[0,1]\times J}) |[0,1]\times L|
  \leq |([0,1]\times L) \isect \scr E_{[0,1]\times J}^*|
  \leq \frac{1}{2} |[0,1]\times L|,
\end{equation}
for all $[0,1]\times L \in \scr E_{[0,1]\times \widetilde J}$.
Note that all the rectangles in $\scr E_{[0,1]\times \widetilde J}$ are of the form $[0,1]\times L$,
see~\eqref{eq:proof:step-1:first-coordinate}.

%%%%%%%%%%%%%%%% END STEP 1 %%%%%%%%%%%%%%%%

%%%%%%%%%%%%%%%% BEGIN STEP 2 %%%%%%%%%%%%%%%%
\noindent
\begin{minipage}[H]{.7\textwidth}
  \textbf{Case~2: $\bm{I \neq [0,1]}$.}
  The figure on the right depicts the transition from Case~1 to Case~2.
  Here, the stage ordinal $i$ is given by $i = \drindex(I\times J)$ with $I \neq [0,1]$.
  The rectangle $I\times J$ is one of the light rectangles.
  The light rectangles fall into two separate cases, see below.
  Up to~\eqref{eq:step-2:block-basis-collection} both cases are treated in tandem.
\end{minipage}
\begin{minipage}[H]{.3\textwidth}
  \begin{center}
    % \begin{asy}
    %   unitsize(.75cm);
    %   defaultpen(fontsize(10pt));

    %   access "images/haar.asy" as haar;

    %   //define fillcolors (no fillcolor == invisible)
    %   int color_1 = 13;//first color_1 rectangles will be dark
    %   int color_2 = 49 ;//then light up to color_2, then fillcolor==invisible
    %   for(int cnt = 0; cnt < color_1; ++cnt){
    %     haar.fillcolor.push(heavygray);
    %   }
    %   for(int cnt = color_1; cnt < color_2; ++cnt) {
    %     haar.fillcolor.push(mediumgray);
    %   }
      
    %   //define which numbers will be displayed (default = not displayed)
    %   int maxnumber = 0;
    %   for(int cnt = 0; cnt < maxnumber; ++cnt) {
    %     haar.numberingselector.push(true);
    %   }

    %   haar.drawmacroblocks(3,.1);
    % \end{asy}
    \includegraphics{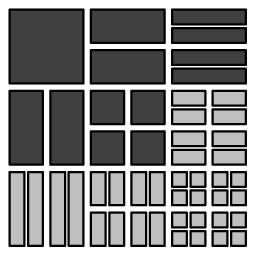}
  \end{center}
\end{minipage}

\ \\
\noindent
We will now construct the collections $\scr Y_{I\times J}$ of $y$--frequencies
and depending on each $y$--frequency $L_0 \in \scr Y_{I\times J}$ the collection
$\scr X_{I\times J}(L_0)$ of $x$--frequencies.
These frequencies will be our building blocks for $\scr E_{I\times J}$ and we have
\begin{equation*}
  \scr E_{I\times J}
  \subset \Union \{\scr X_{I_\times J}(L_0)\times \{L_0\}\, :\, L_0\in \scr Y_{I\times J} \}.
\end{equation*}
The rules by which we finally select $\scr E_{I\times J}$ from the above large union are given
in the equations~\eqref{eq:step-2:block-basis-collection}.

First, let us define the collection $\scr Y_{I\times J}$ simply by putting
\begin{equation*}
  \scr Y_{I\times J} = \{L_0\, :\, [0,1]\times L_0 \in \scr E_{[0,1]\times J}\}.
\end{equation*}
We remark that $K\times L_0 \in \scr E_{[0,1]\times J}$ implies $K = [0,1]$.
Fix $L_0 \in \scr Y_{I\times J}$.
By means of the combinatorial Lemma~\ref{lem:comb-1} we will construct the collection
$\scr X_{I\times J}(L_0)$ of dyadic intervals so that
$\scr X_{I\times J}(L_0)\times \{L_0\}$ is an almost complete cover of $[0,1]\times L_0$ and
simultaneously the rectangles in $\scr X_{I\times J}(L_0)\times \{L_0\}$ have almost vanishing local
frequency weight~\eqref{eq:local_frequency_weight:1}.

Now, let $\scr P$ denote the previous dyadic rectangle indices that are not located
in the same macro block $\scr B_{m,n}$ as $I\times J$, see~\eqref{eq:linear_order-macroblock}.
That is
\begin{equation*}
  \scr P
  = \big\{ I_0\times J_0\, :\,
  I_0\times J_0 \drless I\times J,\, (|I_0|,|J_0|)\neq (|I|,|J|),\, |I_0|\geq |I|,\, |J_0|\geq |J|
  \big\},
\end{equation*}
see Figure~\ref{fig:combinatorial-4}.
\begin{figure}[bt]
  \begin{minipage}[H]{\textwidth}
    \begin{center}
      % \begin{asy}
      % unitsize(1.5cm);
      % defaultpen(fontsize(10pt));

      % access "images/haar.asy" as haar;

      % //draw first picture
      % picture pic_1 = new picture;

      % //color specs
      % haar.fillcolor.delete();
      % int color_1 = 9;//heavygray
      % int color_2 = 13;//invisible
      % int color_3 = 17;//heavygray
      % int color_4 = 25;//mediumgray, then invisisble
      % for(int cnt = 0; cnt < color_1; ++cnt){
      %   haar.fillcolor.push(heavygray);
      % }
      % for(int cnt = color_1; cnt < color_2; ++cnt) {
      %   haar.fillcolor.push(invisible);
      % }
      % for(int cnt = color_2; cnt < color_3; ++cnt) {
      %   haar.fillcolor.push(heavygray);
      % }
      % for(int cnt = color_3; cnt < color_4; ++cnt) {
      %   haar.fillcolor.push(mediumgray);
      % }

      % haar.drawmacroblocks(3,.05);
      % add(pic_1, currentpicture);
      % erase(currentpicture);

      % //draw second picture
      % picture pic_2 = new picture;

      % //color specs
      % haar.fillcolor.delete();
      % int color_1 = 13;//heavygray
      % int color_2 = 25;//invisible
      % int color_3 = 33;//mediumgray
      % for(int cnt = 0; cnt < color_1; ++cnt){
      %   haar.fillcolor.push(heavygray);
      % }
      % for(int cnt = color_1; cnt < color_2; ++cnt) {
      %   haar.fillcolor.push(invisible);
      % }
      % for(int cnt = color_2; cnt < color_3; ++cnt) {
      %   haar.fillcolor.push(mediumgray);
      % }

      % haar.drawmacroblocks(3,.05);
      % add(pic_2, currentpicture);
      % erase(currentpicture);

      % //merge the two drawings
      % add(currentpicture, pic_1);
      % add(currentpicture,shift(5,0)*pic_2);

      % //labels
      % draw("$J_0$",(3.5,-1.05)--(3.5,-.6),E);
      % label("$R_1$",(1.3,-.825), white);
      % label("$R_2$",(1.8,-.825), white);
      % \end{asy}
      \includegraphics{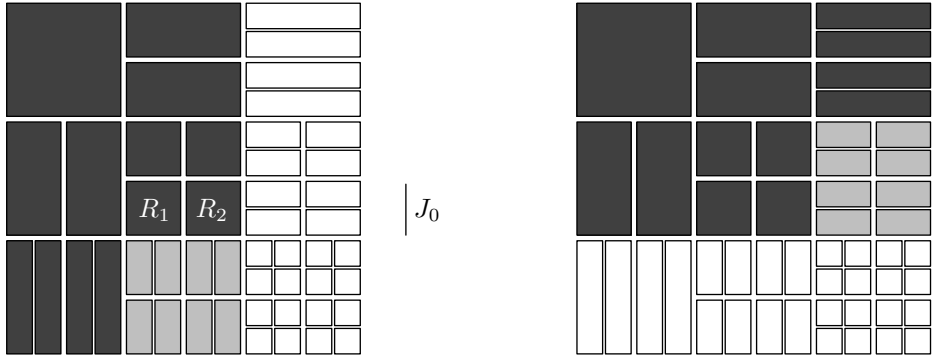}
    \end{center}
    \caption{If $I\times  J$ is one of the shaded rectangles then $\scr P$ is the collection of the
      black rectangles for two generic situations.
      If $L_0'$ is defined by~\eqref{eq:y--low-frequency} then 
      $L_0'\isect L_0 \neq \emptyset$ and $L_0'\subsetneq L_0$ is excluded by choice of $\scr P$.
      The interval $J_0$ defines the strip $\dint_1\times \{J_0\} = \{R_1,R_2\}$ in $\mathbb A$,
      see~\eqref{eq:index_strip}.
    }
    \label{fig:combinatorial-4}
  \end{minipage}
\end{figure}
For each $I_0\times J_0\in \scr P$ there exists a unique $L_0'\in \scr Y_{I_0\times J_0}$ such that
$L_0'\isect L_0 \neq \emptyset$ \textbf{in which case} $L_0' \supset L_0$, 
see Figure~\ref{fig:combinatorial-4}.
We display the logical dependence by writing
\begin{equation}\label{eq:y--low-frequency}
L_0' = L_0'(I\times J, L_0, I_0\times J_0).
\end{equation}

Next, we further partition $\scr P$ into strip collections.
Recall that $\dint_m = \{ I\in \dint\, :\, |I|=2^{-m}\}$.
We define $\mathbb A$ by the following rule:
if $I_0\times J_0\in \scr P$ and $|I_0| = 2^{-m}$ we put $\dint_m\times \{J_0\}\in \mathbb A$.
In other words
\begin{equation}\label{eq:index_strip}
  \mathbb A = \big\{ \dint_m\times \{J_0\}\, :\,
  m\in \bb N_0,\, I_0\times J_0\in \scr P,\,  I_0 \in \dint_m
  \big\}.
\end{equation}
See Figure~\ref{fig:combinatorial-4}.
Note that if $\scr S\in \mathbb A$ then there exist $m$ and $J_0\in \dint$ such that
$\scr S = \dint_m\times \{J_0\}$ and for each $I_0'\in \dint_m$ we have
$I_0'\times J_0 \in \scr P$.
We have clearly $\scr P = \Union_{\scr S\in \mathbb A} \scr S$.

After the preperatory step above we now turn to the core construction, which uses
$\scr X_{I_0\times J_0}(L_0')$, $I_0\times J_0\in \scr P$ as input and returns the collection
$\scr X_{I\times J}(L_0)$ as output.
We extract the relevant information carried by the input collection by defining the following sets:
\begin{equation}\label{eq:x-supports}
  W_{I\times J}(\scr S, L_0)
  = \Union \big\{ \scr X_{I_0\times J_0}^*(L_0')\, :\, I_0\times J_0\in \scr S \big\},
  \qquad \scr S \in \mathbb A.
\end{equation}
(We emphasize the logical dependece $L_0' = L_0'(I\times J, L_0, I_0\times J_0)$.)
See Figure~\ref{fig:combinatorial-5}.
\begin{figure}[bt]
  \begin{minipage}[H]{\textwidth}
    \begin{center}
      \includegraphics{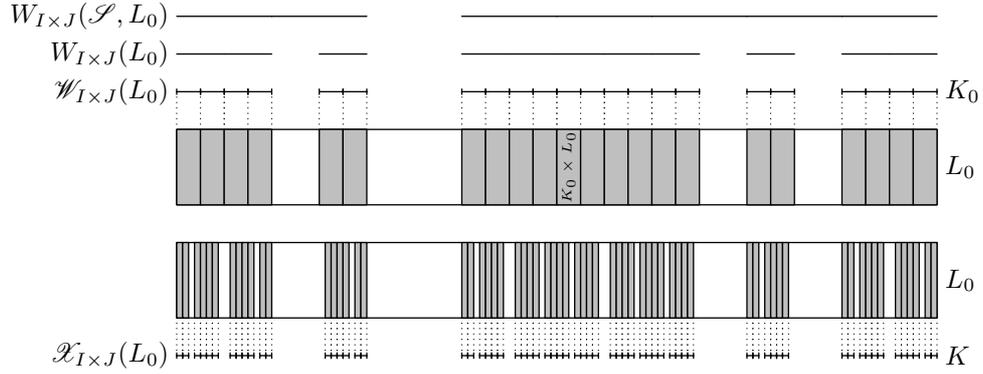}
    \end{center}
    \caption{From top to bottom the figure describes the transition from the input data
      $W_{I\times J}(\scr S,L_0)$ to their intersection $W_{I\times J}(L_0)$ to the fine covering
      $\scr W_{I\times J}(L_0)$ and the rectangles $\scr W_{I\times J}(L_0)\times \{L_0\}$.
      The refinement $\scr X_{I\times J}(L_0)\times \{L_0\}$ of
      $\scr W_{I\times J}(L_0)\times \{L_0\}$ results from Lemma~\ref{lem:comb-1}, hence
      $(K_0\isect \scr X_{I\times J}(L_0))\times \{L_0\}$ almost covers $K_0\times L_0$.}
    \label{fig:combinatorial-5}
  \end{minipage}
\end{figure}
Now we take the intersection over the strips $\scr S\in \mathbb A$
\begin{equation*}
  W_{I\times J}(L_0) = \Isect_{\scr S\in \mathbb A}
  W_{I\times J}(\scr S, L_0).
\end{equation*}
We next choose a fine covering of $W_{I\times J}(L_0)$ by intervals of equal length.
To this we put
\begin{equation*}
  \eta_i = \frac{1}{2} \min\{|K|\, :\,
  \exists L,\, K\times L \in \Union_{I_0\times J_0 \in \scr P} \scr E_{I_0\times J_0}
  \},
\end{equation*}
and set
\begin{equation*}
  \scr W_{I\times J}(L_0) = \{K \in \dint\, :\, |K| = \eta_i,\, K \subset W_{I\times J}(L_0) \}.
\end{equation*}
The collection of intervals $\scr W_{I\times J}(L_0)$ gives rise to the collection of pairwise
disjoint dyadic rectangles $\scr W_{I\times J}(L_0)\times \{L_0\}$, see
Figure~\ref{fig:combinatorial-5}.
By means of the combinatorial Lemma~\ref{lem:comb-1} we refine this collection of rectangles and
obtain a almost complete covering of $W_{I\times J}(L_0)\times \{L_0\}$ consisting of rectangles
$K\times L_0$ with almost vanishing local frequency weight $f_{i-1}$ specified below.
It is important that in the refined covering the $y$--component $L_0$ remains intact.

We defined previously that
\begin{equation*}
  b_j = \sum_{K\times L \in \scr E_j} h_{K\times L},
  \qquad 1 \leq j \leq i-1.
\end{equation*}
Since the collection $\scr E_j$ consists of pairwise disjoint rectangles
we have by~\eqref{eq:definition_hardys_norm} and duality that
\begin{equation*}
  \|b_j\|_{\bmos} = 1
  \qquad\text{and}\qquad
  \|b_j\|_{\hardys} = |\scr E_j^*|.
\end{equation*}
Now, put
\begin{equation}\label{eq:proof:step-2:building_block_size}
  \beta_i = \min\{|K_0\times L_0|\, :\, L_0\in \scr Y_{I\times J}, K_0\in \scr W_{I\times J}(L_0)\}
\end{equation}
and define for all $1 \leq j \leq i-1$
\begin{align}
  x_j &:= \frac{1}{i-1} T^*b_j,
  &y_j &:= \frac{\beta_i}{(i-1) |\scr E_j^*|} T b_j,
  \label{eq:proof:step-2:functions}
\end{align}
Recall that we are using the normalization $\|T\| \leq 1$, hence
\begin{equation*}
  \sum_{j=1}^{i-1} \|x_j\|_{\bmos} \leq 1
  \qquad\text{and}\qquad
  \sum_{j=1}^{i-1} \|y_j\|_{\hardys} \leq \beta_i.
\end{equation*}
We next define the local frequency weight
\begin{equation}\label{eq:local_frequency_weight:2}
  f_{i-1}(K\times L)
  = \sum_{j=1}^{i-1} |\langle x_j, h_{K\times L} \rangle| + |\langle y_j, h_{K\times L}\rangle|,
  \qquad K\times L\in \drec.
\end{equation}
We fix $K_0\in \scr W_{I\times J}(L_0)$ and let
\begin{equation*}
  \scr K(K_0\times L_0) = \{K\times L_0\, :\,
  K \subsetneq K_0,\,
  f_{i-1}(K\times L_0)
  \leq \tau_i |K\times L_0|
  \},
\end{equation*}
where the constant $\tau_i$ is given by
\begin{equation}\label{eq:proof:step-2:constant_definition}
  \tau_i
  = \frac{2^{-i}}{4(i-1)} \beta_i \min_{j\leq i} \varepsilon_j \|b_j\|_2^2.
\end{equation}
Applying Lemma~\ref{lem:comb-1} to $\scr K(K_0\times L_0)$ yields an integer
$k = k(K_0\times L_0)$ such that
\begin{equation}\label{eq:step-2:integer-bound}
  1 \leq k(K_0\times L_0) < \frac{(i-1)^2}{\delta_i^2\, \tau_i^2} + 1
\end{equation}
and so that
\begin{equation*}
  \scr Z_{I\times J}(K_0\times L_0)
  = \big\{ K\times L_0 \in \scr K(K_0\times L_0)\, :\, |K| = 2^{-k(K_0\times L_0)} |K_0| \big\}
\end{equation*}
satisfies
\begin{equation}\label{eq:step-2:local-measure-estimate}
  (1-\delta_i) |K_0 \times L_0|
  \leq |\scr Z_{I\times J}^*(K_0\times L_0)|
  \leq |K_0 \times L_0|.
\end{equation}
Finally, the result of our construction is thus
\begin{equation}\label{eq:step-2:local-x-frequency}
  \scr X_{I\times J}(L_0)
  = \{K\, :\, K\times L_0 \in \scr Z_{I\times J}(K_0\times L_0), K_0 \in \scr W_{I\times J}(L_0)
  \},
\end{equation}
and
\begin{equation}\label{eq:step-2:building-blocks}
  \scr Z_{I\times J} = \Union \big\{ \scr X_{I\times J}(L_0)\times \{L_0\}\, :\,
  L_0\in \scr Y_{I\times J}
  \big\}.
\end{equation}
Observe that the following identity holds:
\begin{equation*}
  \scr Z_{I\times J} = \Union \big\{ \scr Z_{I\times J}(K_0\times L_0)\, :\,
  L_0 \in \scr Y_{I\times J},\, K_0\in \scr W_{I\times J}(L_0)
  \big\}.
\end{equation*}
Since $\scr Z_{I\times J}(K_0 \times L_0) \subset \scr K(K_0\times L_0)$, we have the estimate
\begin{equation}\label{eq:step-2:quasi-diag-estimate}
  f_{i-1}(K\times L_0)
  \leq \tau_i |K\times L_0|,
  \qquad\text{for all $K\times L_0 \in \scr Z_{I\times J}$.}
\end{equation}

Up to this point, the construction for Case~2.a and
Case~2.b are identical.
Now is the time to distinguish between the cases $J = [0,1]$ and $J \neq [0,1]$.

\begin{subequations}\label{eq:step-2:block-basis-collection}
  \noindent
  \begin{minipage}[H]{.7\linewidth}
    \textbf{Case~2.a: $\bm{I \neq [0,1]}$, $\bm{J = [0,1]}$.}
    The light rectangles $I\times J$ are the ones to which Case~2.a is applicable.
    The collection $\scr E_{I\times J}$ is defined in~\eqref{eq:step-2:block-basis-collection-2.a}.
    The collections $\scr E_{I_0\times J_0}$ indexed by the dark rectangles $I_0\times J_0$ are
    already well defined.
    The white ones are ignored.
  \end{minipage}
  \begin{minipage}[H]{.25\linewidth}
  \begin{center}
    % \begin{asy}%[height=5cm,width=\the\linewidth,inline=true]
    %   unitsize(.75cm);
    %   defaultpen(fontsize(10pt));

    %   access "images/haar.asy" as haar;

    %   //define fillcolors (no fillcolor == invisible)
    %   int color_1 = 13;//first color_1 rectangles will be dark
    %   int color_2 = 17 ;//then light up to color_2, then fillcolor==invisible
    %   for(int cnt = 0; cnt < color_1; ++cnt){
    %     haar.fillcolor.push(heavygray);
    %   }
    %   for(int cnt = color_1; cnt < color_2; ++cnt) {
    %     haar.fillcolor.push(mediumgray);
    %   }
      
    %   //define which numbers will be displayed (default = not displayed)
    %   int maxnumber = 0;
    %   for(int cnt = 0; cnt < maxnumber; ++cnt) {
    %     haar.numberingselector.push(true);
    %   }

    %   haar.drawmacroblocks(3,.1);
    % \end{asy}
    \includegraphics{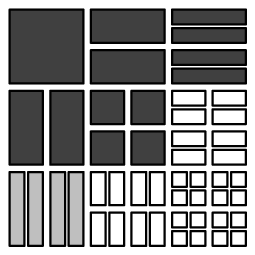}
  \end{center}
\end{minipage}

\noindent
We are now ready to define $\scr E_{I\times J}$ using the Gamlen-Gaudet procedure.
To this end recall first that $\widetilde I$ denotes the unique dyadic interval satisfying
$\widetilde I \supset I$ and $|\widetilde I| = 2|I|$.
Second, for a dyadic interval $K_0$ we denote its left half by
$K_0^\ell$ ($K_0^\ell \subset K_0$, $|K_0^\ell| = |K_0|/2$, $\inf K_0^\ell = \inf K_0$)
and its right half by $K_0^r$ ($K_0^r = K_0 \setminus K_0^\ell$).
We define the sets
\begin{equation*}
  E_{\widetilde I \times J}^\ell
  = \Union_{K_0\times L_0 \in \scr E_{\widetilde I\times J}} K_0^\ell\times L_0
  \quad\text{and}\quad
  E_{\widetilde I \times J}^r
  = \Union_{K_0\times L_0 \in \scr E_{\widetilde I\times J}} K_0^r\times L_0.
\end{equation*}
If $I$ is the left half of $\widetilde I$, then we put
\begin{equation}\label{eq:step-2:block-basis-collection-2.a}
  \scr E_{I\times J} = \{ K\times L_0 \in \scr Z_{I\times J}\, :\,
    K\times L_0 \subset  E_{\widetilde I\times J}^\ell
  \},
\end{equation}
Alternatively, if $I$ is the right half of $\widetilde I$, then 
\begin{equation}\label{eq:step-2:block-basis-collection-2.b}
  \scr E_{I\times J} = \{ K\times L_0 \in \scr Z_{I\times J}\, :\,
    K\times L_0 \subset  E_{\widetilde I\times J}^r
  \}.
\end{equation}

\begin{minipage}[H]{.7\linewidth}
  \textbf{Case~2.b: $\bm{I \neq [0,1]}$, $\bm{J \neq [0,1]}$.}
  The figure on the right depicts the transition from Case~2.a to Case~2.b.
  The light rectangles $I\times J$ are the ones covered by Case~2.b.
  The collection $\scr E_{I\times J}$ is defined in~\eqref{eq:step-2:block-basis-collection-2.c}.
  The collections $\scr E_{I_0\times J_0}$ indexed by the dark rectangles $I_0\times J_0$ are well
  defined before the first light rectangle is treated.
  \end{minipage}
  \begin{minipage}[H]{.25\linewidth}
    \begin{center}
      % \begin{asy}%[height=5cm,width=\the\linewidth,inline=true]
      %   unitsize(.75cm);
      %   defaultpen(fontsize(10pt));

      %   access "images/haar.asy" as haar;

      %   //define fillcolors (no fillcolor == invisible)
      %   int color_1 = 17;//first color_1 rectangles will be dark
      %   int color_2 = 49 ;//then light up to color_2, then fillcolor==invisible
      %   for(int cnt = 0; cnt < color_1; ++cnt){
      %     haar.fillcolor.push(heavygray);
      %   }
      %   for(int cnt = color_1; cnt < color_2; ++cnt) {
      %     haar.fillcolor.push(mediumgray);
      %   }
        
      %   //define which numbers will be displayed (default = not displayed)
      %   int maxnumber = 0;
      %   for(int cnt = 0; cnt < maxnumber; ++cnt) {
      %     haar.numberingselector.push(true);
      %   }

      %   haar.drawmacroblocks(3,.1);
      % \end{asy}
      \includegraphics{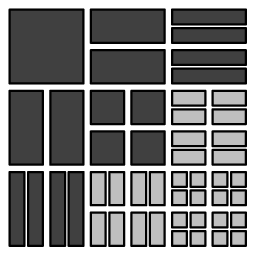}
    \end{center}
  \end{minipage}
  \begin{equation}\label{eq:step-2:block-basis-collection-2.c}
    \scr E_{I\times J} = \{ K\times L_0 \in \scr Z_{I\times J}\, :\,
    K\times L_0 \subset \scr E_{I\times \widetilde J}^* \}
  \end{equation}
\end{subequations}
(A comment on~\eqref{eq:step-2:block-basis-collection-2.c}: It is here where we bring in the
combinatorial harvest of Lemma~\ref{lem:comb-1} where we insisted that the coverings leave the
$y$--components $L_0$ intact, see Figure~\ref{fig:combinatorial-2}.
Moreover, the definition~\eqref{eq:step-2:block-basis-collection-2.c} would not  be possible if we
used fragmented coverings as depicted in Figure~\ref{fig:combinatorial-3}.)

In each of the above cases~\eqref{eq:step-2:block-basis-collection} we put
\begin{equation}\label{eq:block-basis-definition}
  b_{I\times J} = \sum_{K\times L\in \scr E_{I\times J}} h_{K\times L}.
\end{equation}

\subsection*{A first property of $\scr E_{I\times J}$}\hfill\\

We have now completed the construction part of the proof.
Before we turn to a detailed examination of the entire system
$\{\scr E_{I\times J}\, :\, I\times J\in\drec \}$ and $\{b_{I\times J}\, :\, I\times J\in\drec \}$
we analyze the intersections $K\times L\isect \scr E_{I\times J}^*$ where
$K\times L \in \scr E_{\widetilde I\times J}\union \scr E_{I\times \widetilde J}$.
Put
\begin{equation*}
  1-\alpha_{I\times J} = \prod_{I_0\times J_0 \drlesseq I\times J} (1 - \delta_{I_0\times J_0})^2.
\end{equation*}
We claim that
\begin{subequations}\label{eq:step-2:global-measure-estimate}
  \begin{equation}\label{eq:step-2:global-measure-estimate-a}
    \frac{1}{2} (1-\alpha_{I\times J}) |K\times L|
    \leq |(K\times L) \isect \scr E_{I\times J}^*|
    \leq \frac{1}{2} |K\times L|,
  \end{equation}
  for all $K\times L \in \scr E_{\widetilde I\times J}$ if $I\neq [0,1]$, as well as
  \begin{equation}\label{eq:step-2:global-measure-estimate-b}
    \frac{1}{2} (1-\alpha_{I\times J}) |K\times L|
    \leq |(K\times L) \isect \scr E_{I\times J}^*|
    \leq \frac{1}{2} |K\times L|,
  \end{equation}
  for all $K\times L \in \scr E_{I\times \widetilde J}$ if $J\neq [0,1]$.
\end{subequations}

Indeed, we only have to verify the left hand side estimates.
First, let $K\times L \in \scr E_{\widetilde I\times J}$.
Observe that since $\scr Y_{\widetilde I\times J} = \scr Y_{I\times J}$ and
$|\scr E_{I\times J}^*\isect ([0,1]\times L)|
= \frac{1}{2} |\scr X_{I\times J}^*(L)\times L|$ for all $L\in \scr Y_{I\times J}$,
we have
\begin{equation}\label{eq:step-2:global-measure-estimate-proof-1}
  |(K\times L) \isect \scr E_{I\times J}^*|
  = \frac{1}{2} |(K\times L) \isect (\scr X_{I\times J}^*(L)\times L)|.
\end{equation}
Obviously, by~\eqref{eq:step-2:local-measure-estimate}, the right hand side is larger than
\begin{equation}\label{eq:step-2:global-measure-estimate-proof-2}
  \frac{1}{2} (1-\delta_{I\times J}) |K\isect W_{I\times J}(L)|\, |L|.
\end{equation}
We go back over the course by which we have come and see that
\begin{equation}\label{eq:step-2:global-measure-estimate-proof-3}
  |K\isect W_{I\times J}(L)|
  \geq \prod_{\widetilde I\times J \drless I_0\times J_0 \drless I\times J}
    (1 - \delta_{I_0\times J_0}) |K|.
\end{equation}
Combining~\eqref{eq:step-2:global-measure-estimate-proof-1},
with~\eqref{eq:step-2:global-measure-estimate-proof-2}
and~\eqref{eq:step-2:global-measure-estimate-proof-3}
yields~\eqref{eq:step-2:global-measure-estimate-a}.
Second, let $K\times L \in \scr E_{I\times \widetilde J}$ and $J\neq [0,1]$.
By the definition of $\scr E_{I\times J}$ and~\eqref{eq:step-2:building-blocks} we have
\begin{equation}\label{eq:step-2:global-measure-estimate-proof-4}
  |(K\times L) \isect \scr E_{I\times J}^*|
  = \sum_{L_0 \in \scr Y_{I\times J}} |(K\times L) \isect (\scr X_{I\times J}^*(L_0)\times L_0)|.
\end{equation}
For each summand note the identity
\begin{equation}\label{eq:step-2:global-measure-estimate-proof-5}
  |(K\times L) \isect (\scr X_{I\times J}^*(L_0)\times L_0)|
  = |K \isect \scr X_{I\times J}^*(L_0)|\, |L\isect L_0|.
\end{equation}
As before, we have
\begin{equation}\label{eq:step-2:global-measure-estimate-proof-6}
  |K \isect \scr X_{I\times J}^*(L_0)|
  \geq (1-\delta_{I\times J}) |K \isect W_{I\times J}(L_0)|
\end{equation}
and
\begin{equation}\label{eq:step-2:global-measure-estimate-proof-7}
  |K\isect W_{I\times J}(L_0)|
  \geq \prod_{I\times\widetilde J \drless I_0\times J_0 \drless I\times J}
    (1 - \delta_{I_0\times J_0}) |K|.
\end{equation}
Next, we observe that by~\eqref{eq:step-2:global-measure-estimate-proof-5},
\eqref{eq:step-2:global-measure-estimate-proof-6} and
\eqref{eq:step-2:global-measure-estimate-proof-7}, the sum in the right hand side
of~\eqref{eq:step-2:global-measure-estimate-proof-4} is larger than
\begin{equation}\label{eq:step-2:global-measure-estimate-proof-8}
  \Big( \prod_{I\times \widetilde J\drless I_0\times J_0 \drlesseq I\times J}
    (1 - \delta_{I_0\times J_0})\Big) |K| \sum_{L_0 \in \scr Y_{I\times J}} |L\isect L_0|.
\end{equation}
Taking into account that $J \subset \widetilde J$, the Gamlen-Gaudet construction of
Case~1 gives
\begin{equation}\label{eq:step-2:global-measure-estimate-proof-9}
  \sum_{L_0 \in \scr Y_{I\times J}} |L\isect L_0|
  \geq \frac{1}{2} (1-\delta_{[0,1]\times J}) |L|.
\end{equation}
Finally, combining~\eqref{eq:step-2:global-measure-estimate-proof-8}
and~\eqref{eq:step-2:global-measure-estimate-proof-9}
with~\eqref{eq:step-2:global-measure-estimate-proof-4}
yields~\eqref{eq:step-2:global-measure-estimate-b}.

%%%%%%%%%%%%%%%% END STEP 2 %%%%%%%%%%%%%%%%

\subsection*{Essential properties of our construction}\hfill\\

\subsubsection*{Output of the inductive step}\hfill\\
Having completed the construction of $\{\scr E_{I\times J}\, :\, I\times J \in \drec_n\}$ we record
the following crucial properties.
First, \eqref{eq:step-1:global-measure-estimate} and~\eqref{eq:step-2:global-measure-estimate} imply
that for each $I_0\times J_0,\, I_1\times J_1 \in \drec_n$ such that
$I_0 \supset I_1,\, J_0 \supset J_1$ and $|I_0\times J_0| = 2\, |I_1\times J_1|$ we have
\begin{equation}\label{eq:summary:measure-estimate}
  \frac{1}{2} \prod_{I\times J \drlesseq I_1\times J_1}
  (1 - \delta_{I\times J})^2 |K\times L|
  \leq |(K\times L) \isect \scr E_{I_1\times J_1}^*|
  \leq \frac{1}{2} |K\times L|,
\end{equation}
for all $K\times L \in \scr E_{I_0\times J_0}$.
Second, \eqref{eq:proof:step-1:building_block_size}, \eqref{eq:proof:step-1:functions},
\eqref{eq:step-1:quasi-diag-estimate} and~\eqref{eq:step-1:block-basis-collection} as well
as~\eqref{eq:proof:step-2:building_block_size}, \eqref{eq:proof:step-2:functions},
\eqref{eq:step-2:quasi-diag-estimate} and~\eqref{eq:step-2:block-basis-collection} imply
\begin{equation}\label{eq:quasi-diag-estimate}
  \sum_{j=1}^{i-1} |\langle T^* b_j, h_{K\times L} \rangle| + |\langle Tb_j, h_{K\times L}
  \rangle|
  \leq \frac{(i-1)\, \tau_i}{\beta_i} |K\times L|,
\end{equation}
for all $i \in \bb N$ and $K\times L \in \scr E_i$.
Recall that $\scr E_i = \scr E_{I\times J}$ provided $i = \drindex(I\times J)$.

\subsubsection*{Bi-tree property}\hfill\\
The collection $\{\scr E_{I\times J}^*\, :\, I\times J \in \drec_n\}$ forms a bi-tree,
see~\eqref{eq:bi-tree}.
The bi-tree constant is determined by the local product structure~\eqref{eq:thm:quasi-diag:13}
verified below.
In particular
\begin{equation}\label{eq:bi-tree-estimate}
  \frac{1}{2} |I\times J| \leq |\scr E_{I\times J}^*| \leq |I\times J|.
\end{equation}

\subsubsection*{The local product structure of $\scr E_{I\times J}$}\hfill\\
Here, we exploit our choice of the constants $\delta_{I\times J}$,
see~\eqref{eq:constant_definition}.
We carry over~\eqref{eq:summary:measure-estimate} to each pair of nested dyadic rectangles.
Let $I_0\times J_0,\, I_1\times J_1 \in \drec$ such that
$I_0 = \pred^{i}(I_1)$ and $J_0 = \pred^j(J_1)$ for some $i,j \in \bb N_0$.
Then, iterating~\eqref{eq:summary:measure-estimate} yields
\begin{equation}\label{eq:thm:quasi-diag:13}
  \frac{1}{2}\, |K\times L|
  \leq 2^{i+j} |(K\times L) \isect \scr E_{I_1\times J_1}^*|
  \leq |K\times L|,
\end{equation}
for all $K\times L \in \scr E_{I_0\times J_0}$.
Our construction with its inherent complications permits us now verify the crucial
estimate~\eqref{eq:thm:quasi-diag:13}.
We present only the proof for the lower estimate since the verification of the upper estimate
follows the same line of reasoning.
Let $I_0\times J_0$ and $I_1\times J_1$ be a nested pair of dyadic rectangles as specified above.
We now define a path $p(I_0\times J_0, I_1\times J_1)$ of nested rectangles
$I^{(m)}\times J^{(m)}$ connecting $I_1\times J_1$ to $I_0\times J_0$ as follows.
We define $I^{(0)} = I_1$, $I^{(i+j)} = I_0$ and $J^{(0)} = J_1$, $J^{(i+j)} = J_0$ as well as
\begin{align*}
  I^{(m+1)} & = \widetilde I^{(m)}
  \quad\text{and}\quad
  J^{(m+1)} = J^{(m)},
  &\text{if\quad $0 \leq m \leq i-1$},\\
  I^{(m+1)} & = I^{(m)}
  \quad\text{and}\quad
  J^{(m+1)} = \widetilde J^{(m)},
  &\text{if\quad $i \leq m \leq i+j-1$}.
\end{align*}
Iterating the local property~\eqref{eq:summary:measure-estimate} along the path
$p=p(I_0\times J_0, I_1\times J_1)$ we obtain
\begin{equation*}
  |(K\times L) \isect \scr E_{I_1\times J_1}^*|
  \geq 2^{-(i+j)} \prod_{I\times J \in p} (1-\alpha_{I\times J}) |K\times L|,
\end{equation*}
where we put
\begin{equation*}
  1-\alpha_{I\times J} = \prod_{k \leq \drindex(I\times J)} (1-\delta_k)^2.
\end{equation*}
Since the length of the path $p$ is at most $2n$, we obtain
\begin{equation*}
  |(K\times L) \isect \scr E_{I_1\times J_1}^*|
  \geq 2^{-(i+j)} |K\times L| \Big(1 - 4n \sum_{k=1}^\infty  \delta_k\Big).
\end{equation*}
As we specified $\delta_k = 2^{-k}/(8 n)$ in~\eqref{eq:constant_definition} we see that
\eqref{eq:thm:quasi-diag:13} holds.

\subsection*{The boundedness of the orthogonal projection $Q$}\hfill\\

The collections of dyadic rectangles $\scr E_{I\times J}$ gives rise to the block basis
\begin{equation*}
  b_{I\times J} = \sum_{K\times L \in \scr E_{I\times J}} h_{K\times L}
\end{equation*}
and the orthogonal projection
\begin{equation*}
  Q(f) = \sum_{I\times J \in \drec_n}
  \big\langle f, \frac{b_{I\times J}}{\|b_{I\times J}\|_2} \big\rangle\,
  \frac{b_{I\times J}}{\|b_{I\times J}\|_2}.
\end{equation*}
Feeding the estimate~\eqref{eq:thm:quasi-diag:13} into Theorem~\ref{pro:projection} we obtain
that
\begin{equation*}
  \|Q\, :\, \hardys \to \hardys\| \leq C_2,
\end{equation*}
for some universal constant $C_2 > 0$.

\subsection*{The basis $\{b_i\}$ are almost eigenvectors for $T$}\hfill\\

\noindent
We show that we have
\begin{equation*}
  T b_i = \frac{|\langle T b_i, b_i \rangle|}{\|b_i\|_2^2} b_i +
  \text{tiny error}.
\end{equation*}
To be more precise, we claim that
\begin{equation}\label{eq:almost-eigenvectors:0}
  \sum_{j\, :\, j \neq i} | \langle T b_{j}, b_i \rangle |
  \leq \varepsilon_i \|b_{i}\|_2^2,
  \qquad \text{for all $i$}.
\end{equation}

We begin the proof of~\eqref{eq:almost-eigenvectors:0} by summing
estimate~\eqref{eq:quasi-diag-estimate} over all $K\times L \in \scr E_i$ to obtain
\begin{equation}\label{eq:quasi-diag-estimate-a}
  \sum_{j=1}^{i-1} |\langle T^* b_j, b_i \rangle| + |\langle Tb_j, b_i \rangle|
  \leq \frac{(i-1)\, \tau_i}{\beta_i} \|b_i\|_2^2.
\end{equation}
Reversing the roles of $i$ and $j$ in~\eqref{eq:quasi-diag-estimate-a} gives
\begin{equation}\label{eq:almost-eigenvectors:2}
  |\langle T b_j, b_i \rangle| = |\langle b_j, T^* b_i\rangle|
  \leq \frac{(j-1)\, \tau_j}{\beta_j} \|b_j\|_2^2,
  \qquad j \geq i+1.
\end{equation}
Taking the sum in~\eqref{eq:almost-eigenvectors:2} and adding~\eqref{eq:quasi-diag-estimate-a} we
get
\begin{equation}\label{eq:quasi-diag-estimate-b}
  \sum_{j\, :\, j\neq i} |\langle T b_j, b_i \rangle|
  \leq 2\sum_{j \geq i} \frac{(j-1)\, \tau_j}{\beta_j} \|b_j\|_2^2.
\end{equation}
Now, recall that in~\eqref{eq:proof:step-1:constant_definition} and~\eqref{eq:proof:step-2:constant_definition} we defined the constants $\tau_j$ by
\begin{equation}\label{eq:almost-eigenvector:constant_definition}
  \tau_j
  = \frac{2^{-j}}{4(j-1)} \beta_j \min_{k\leq j} \|b_k\|_2^2 \varepsilon_k
\end{equation}
Finally, plugging~\eqref{eq:almost-eigenvector:constant_definition} into~\eqref{eq:bi-tree-estimate}
yields
\begin{equation*}
  \sum_{j\, :\, j\neq i} |\langle T b_j, b_i \rangle|
  \leq 2^{-i} \varepsilon_i \|b_i\|_2^2,
\end{equation*}
which is certainly smaller than estimate~\eqref{eq:almost-eigenvectors:0}.
%%%%%%%%%%%%%%%% END LOCAL %%%%%%%%%%%%%%%%

%%%%%%%%%%%%%%%% BEGIN FACTORIZATION %%%%%%%%%%%%%%%%
\section{Localization in bi--parameter BMO}\label{s:factorization}

\noindent
In this section we give the proof of our main theorem~\ref{thm:bmo-primary} restated below.

\begin{thm*}[\textbf{Main theorem~\ref{thm:bmo-primary}}]%\label{thm:bmo-primary}
  For any operator
  \begin{equation*}
    T\, :\, \bmos\rightarrow \bmos
  \end{equation*}
  the identity on $\bmos$ factors through
  $H = T$ or $H = \Id-T$, that is
  \begin{equation}\label{eqn:primarity-bmo}
    \vcxymatrix{\bmos \ar[r]^\Id \ar[d]_E & \bmos\\
      \bmos \ar[r]_H & \bmos \ar[u]_P}
    \qquad \|E\|\|P\| \leq C.
  \end{equation}
\end{thm*}

\noindent
The structure of the proof given below follows the general localization method introduced by
Bourgain~\cite{bourgain:1983} to treat factorization problems.
We first list the basic steps of the argument:
\begin{enumerate}[(i)]
\item We exploit Wojtaszczyk's isomorphism asserting that
  \begin{equation*}
    \bmos \sim \Big( \sum_n \bmos[n] \Big)_\infty.
  \end{equation*}

\item We reduce the factorization problem to the case where the operator $T$ is a diagonal operator
  on $\Big( \sum_n \bmos[n] \Big)_\infty$.

\item We invoke our finite dimensional factorization Theorem~\ref{thm:local-hardy} to infer that in
  fact Theorem~\ref{thm:bmo-primary} holds true for diagonal operators.
\end{enumerate}

We say an operator
$D\, :\, \big( \sum_n \bmos[n] \big)_\infty \longrightarrow \big( \sum_n \bmos[n] \big)_\infty$ is a
diagonal operator if there exists a sequence of operators
$A_n\, :\, \bmos[n] \rightarrow \bmos[n]$ such that
\begin{equation*}
  D(f_1,f_2,\ldots,f_n,\ldots) = (A_1 f_1,A_2 f_2,\ldots,A_n f_n,\ldots).
\end{equation*}
The following theorem provides the reduction to diagonal operators.

\begin{thm}\label{thm:product-diag-bmo}
  For any linear operator
  $T : \big( \sum_n \bmos[n] \big)_\infty\to \big( \sum_n \bmos[n] \big)_\infty$ there exists a
  diagonal operator
  \begin{equation*}
    D\, :\, \big( \sum_n \bmos[n] \big)_\infty \longrightarrow \big( \sum_n \bmos[n] \big)_\infty
  \end{equation*}
  and bounded linear operators
  \begin{equation*}
    R,E : \big( \sum_n \bmos[n] \big)_\infty\to \big( \sum_n \bmos[n] \big)_\infty
  \end{equation*}
  such that
  \begin{equation}\label{eq:product-diag-bmo}
    D = RTE
    \qquad\text{and}\qquad
    \Id - D = R(\Id - T)E.
  \end{equation}
\end{thm}
\noindent
We remark that~\eqref{eq:product-diag-bmo} implies $RE = \Id$.

The proof of Theorem~\ref{thm:product-diag-bmo} relies on the repeated application of the following
theorem which is a simplified variant of Theorem~\ref{thm:quasi-diag}.
\begin{thm}\label{thm:annihil-bmo}
  Let $n \in \mathbb N$ and $\varepsilon > 0$, then there exists an $N=N(n,\varepsilon)$ so that the
  following holds.
  For any $n$--dimensional subspace $F \subset \bmos[N]$, there exists a block--basis
  $\{b_{I\times J}\}$ in $\bmos[N]$ satisfying the following conditions.
  \begin{enumerate}[(i)]
  \item The map $S\, :\, \bmos[n] \rightarrow \bmos[N]$ defined as the linear extension of
    $h_{I\times J} \mapsto b_{I\times J}$ satisfies
    \begin{equation*}
      \frac{1}{C}\, \| f \| \leq \| S f \| \leq C\, \| f \|,
    \end{equation*}
    with universal constant $C$.

  \item The orthogonal projection $Q\, :\, \bmos[N] \rightarrow \bmos[N]$ given by
    \begin{equation*}
      Q(f) = \sum_{I\times J \in \drec_n}
      \big\langle f, \frac{b_{I\times J}}{\|b_{I\times J}\|_2} \big\rangle\,
      \frac{b_{I\times J}}{\|b_{I\times J}\|_2}
    \end{equation*}
    is bounded by
    \begin{equation*}
      \|Q f\|_{\bmos[N]} \leq C \|f\|_{\bmos[N]},
      \qquad f \in \bmos[N],
    \end{equation*}
    for some universal constant $C$ and almost annihilates the space $F$,
    \begin{equation}\label{eq:thm:annihil-bmo}
      \| Q f \|_{\bmos[N]} \leq \varepsilon\, \|f\|_{\bmos[N]},
      \qquad f\in F.
    \end{equation}
  \end{enumerate}
\end{thm}

\begin{proof}
  The proof of Theorem~\ref{thm:annihil-bmo} is a repetition of the almost--diagonalization argument
  in the proof of Theorem~\ref{thm:quasi-diag}, where
  condition~\eqref{eq:thm:annihil-bmo} is simpler to realize than~\eqref{thm:quasi-diag-iii}.
  The situation is analogous to the one parameter case treated in~\cite[290--291]{mueller:2005}.
\end{proof}

\begin{proof}[Proof of Theorem~\ref{thm:product-diag-bmo}]
  The proof of Theorem~\ref{thm:product-diag-bmo} is quantitative and finite dimensional in nature.
  The estimates pertaining specifically to bi--parameter BMO are provided by
  Theorem~\ref{thm:annihil-bmo}.
  The reduction procedure itself is analogous to the corresponding localization theorems
  in~\cite{blower:1990, bourgain:1983, mueller:2005, wark:2007}.

  Let $\varepsilon > 0$ and $\varepsilon_n = 4^{-n-3}\varepsilon$.
  Subsequently, we write $N = N(n) = N(n,\varepsilon_n)$ as specified in
  Theorem~\ref{thm:annihil-bmo}.
  We further abbreviate
  \begin{equation*}
    X_n=\bmos[n]
    \qquad\text{and}\qquad
    X = (\sum_n X_n)_\infty.
  \end{equation*}
  Let $p_j\, :\, X \to X_j$ denote the projection onto the $j$--th coordinate.
  Given a subset $\Lambda$ of $\mathbb N$ we define
  $P_\Lambda\, :\, X \to X$ by
  \begin{equation*}
    p_j P_\Lambda x =
    \begin{cases}
      x_j, & \text{if $j\in \Lambda$}\\
      0, & \text{otherwise}
    \end{cases}
    \qquad\text{for all $x = (x_n)_n \in X$, $j \in \mathbb N$}.
  \end{equation*}

  We will now inductively define an increasing sequence of integers $M(n)$, a decreasing sequence of
  infinite subsets $\Lambda_n$ of $\mathbb N$, subspaces $F_n$ of $X_{M(n)}$
  (see~\eqref{eq:product-diag-bmo-0} below), projections $Q_n\, :\, X_{M(n)}\to X_{M(n)}$ and
  isomorphisms $S_n:X_n\to Q(X_{M(n)})$
  such that
  \begin{enumerate}[(i)]
  \item $\|Q_n\| \leq C$ and $\|S_n\| \|S_n^{-1}\| \leq C$,
    \label{enu:product-diag-bmo:1}
  \item $\|Q_n(x)\|_{X_{M(n)}} \leq \varepsilon_n \|x\|_{X_{M(n)}}$ for all $x \in F_n$,
    \label{enu:product-diag-bmo:2}
  \item $M(n) \in \Lambda_n$ and $\min \Lambda_n > M(n-1)$,
    \label{enu:product-diag-bmo:3}
  \item $\|p_{M(n-1)} T P_{\Lambda_n}\| \leq \varepsilon_n$.
    \label{enu:product-diag-bmo:4}
  \end{enumerate}

  We begin the construction by defining $M(1) = 1$, $\Lambda_1 = \mathbb N$, $Q_1 = \Id$ and
  $F_1 = \{0\}$.
  Assume we have completed our construction for all $1 \leq j \leq n-1$.
  We will now choose an infinite subset $\Lambda_n$ of $\Lambda_{n-1}$ such
  that~\eqref{enu:product-diag-bmo:3} and~\eqref{enu:product-diag-bmo:4} are satisfied.
  Since $X_{M(n-1)}$ is finite dimensional it suffices to show that for every $\varphi \in X^*$
  there exists an infinite subset $\Lambda_n$ of $\Lambda_{n-1}$ such that
  \begin{equation*}
    \|\varphi P_{\Lambda_n}\| \leq \varepsilon_n.
  \end{equation*}
  To this end let $\varphi \in X^*$ and $\Gamma = \{k \in \Lambda_{n-1}\, :\, k > M(n-1)\}$.
  Assume that for each infinite subset $\Lambda$ of $\Gamma$ we have that
  $\|\varphi P_\Lambda\| > \varepsilon_n$.
  Partition the infinite set $\Gamma$ into $m$ disjoint infinite sets $\Gamma_1,\ldots,\Gamma_m$ and
  choose $x_1,\ldots,x_m \in X$ with $\|x_j\| = 1$ such that
  $\varphi P_{\Gamma_j} x_j > \varepsilon_n$.
  Observe that the disjointness of the $\Gamma_j$ implies that
  $\big\|\sum_{j=1}^m P_{\Gamma_j} x_j\big\| \leq 1$, thus
  \begin{equation*}
    m \varepsilon_n
    < \sum_{j=1}^m \varphi P_{\Gamma_j} x_j
    \leq \|\varphi\|.
  \end{equation*}
  This gives a contradiction for sufficiently large $m$, showing~\eqref{enu:product-diag-bmo:3}
  and~\eqref{enu:product-diag-bmo:4}.

  Let the projection $Q^{(n-1)}\, :\, X \to X$ be defined by
  \begin{equation*}
    p_j Q^{(n-1)} x =
    \begin{cases}
      Q_k x_k, & \text{if $j=M(k)$ and $j \leq M(n-1)$}\\
      0, & \text{otherwise}
    \end{cases}
  \end{equation*}
  for all $x = (x_k)_k \in X$ and $j \in \mathbb N$.
  Then define the subspace $W_n = T Q^{(n-1)}(X)$ and choose
  $M(n) = \min\{k \in \Lambda_n\, :\, k \geq N(\dim W_n,\varepsilon_n)\}$,
  where $N = N(\dim W_n,\varepsilon_n)$ is the constant appearing in Theorem~\ref{thm:annihil-bmo}.
  We next specify a subspace $F_n$ by putting
  \begin{equation}\label{eq:product-diag-bmo-0}
    F_n = p_{M(n)} W_n.
  \end{equation}

  Theorem~\ref{thm:annihil-bmo} asserts that there exists a projection $Q_n$ and an isomorphism
  $S_n\, :\, X_n\to Q_n(X_{M(n)})$ such that~\eqref{enu:product-diag-bmo:1}
  and~\eqref{enu:product-diag-bmo:2} are satisfied.

  We will now define the maps $I,Q : X \to X$ by
  \begin{equation*}
    p_j I x =
    \begin{cases}
      S_n x_n, & \text{if $j=M(n)$}\\
      0, & \text{otherwise}
    \end{cases}
    \qquad\text{and}\qquad
    p_j Q x =
    \begin{cases}
      Q_n x_n, & \text{if $j=M(n)$}\\
      0, & \text{otherwise}.
    \end{cases}
  \end{equation*}
  for all $x = (x_n)_n \in X$ and $j \in \mathbb N$.
  Define $J : Q(X) \to X$ by
  \begin{equation*}
    J y = (S_n^{-1} y_{M(n)})_n
    \qquad \text{for all $y = (y_n)_n \in Q(X)$}.
  \end{equation*}
  Note that $JQI = \Id$ and that therefore
  \begin{equation}\label{eq:product-diag-bmo-1}
    \widehat T = JQTI
  \end{equation}
  satisfies
  \begin{equation}\label{eq:product-diag-bmo-2}
    \Id - \widehat T = JQ(\Id-T)I
  \end{equation}
  and moreover $\widehat T$ is a small perturbation of a diagonal operator.
  Indeed, define $D\, :\, X \to X$ by $D = (p_n\widehat T p_n)_n$ and observe that $D$ is a bounded
  diagonal operator for which
  \begin{equation}\label{eq:product-diag-bmo-3}
    \|\widehat T - D\| < \varepsilon,
  \end{equation}
  since we chose $\varepsilon_n = 4^{-n-3} \varepsilon$.
  This is a consequence of conditions~\eqref{enu:product-diag-bmo:1}
  to~\eqref{enu:product-diag-bmo:4}.
  A standard perturbation argument shows finally the existence of the operators
  \begin{equation*}
    R,E : X\to X
  \end{equation*}
  such that
  \begin{equation*}
    D = RTE
    \qquad\text{and}\qquad
    \Id - D = R(\Id - T)E.
    \qedhere
  \end{equation*}
\end{proof}

In Theorem~\ref{thm:product-diag-bmo} we provided the reduction of the general factorization
theorem~\ref{thm:bmo-primary} to the case of diagonal operators.
We now turn to the remaining last step: we show that the factorization theorem holds true for
diagonal operators.
\begin{thm}\label{thm:diagonal-bmo}
  Let $D$ be a diagonal operator on $\big( \sum_n \bmos[n] \big)_\infty$.
  Then the identity factors through $H = D$ or $H = \Id - D$, that is
  \begin{equation*}
    \vcxymatrix{\big( \sum_n \bmos[n] \big)_\infty \ar[r]^\Id \ar[d]_E
        & \big( \sum_n \bmos[n] \big)_\infty\\
        \big( \sum_n \bmos[n] \big)_\infty \ar[r]_H
        & \big( \sum_n \bmos[n] \big)_\infty \ar[u]_P}
      \qquad \|E\|\|P\| \leq C,
  \end{equation*}
  where $C$ is a universal constant.
\end{thm}

\begin{proof}
  Let $A_n\, :\, \bmos[n]\rightarrow \bmos[n]$ be the linear map defining the diagonal operator $D$,
  that is
  \begin{equation*}
    D(f_1,f_2,\ldots,f_n,\ldots) = (A_1 f_1,A_2 f_2,\ldots,A_n f_n,\ldots).
  \end{equation*}
  By Theorem~\ref{thm:local-hardy} the identity on $\bmos[n]$ factors through $H_n = A_{N(n)}$ or
  $H_n = \Id - A_{N(n)}$, that is
  \begin{equation*}
    \vcxymatrix{\bmos[n] \ar[r]^\Id \ar[d]_{E_n} & \bmos[n]\\
      \bmos[N] \ar[r]_{H_n} & \bmos[N] \ar[u]_{P_n}}
      \qquad \|E_n\|\|P_n\| \leq C,
  \end{equation*}
  for some universal constant $C > 0$.
  If there exists an infinite sequence $\{k(n)\}$ so that $H_{k(n)} = A_{N(k(n))}$, then the
  identity on $\big( \sum_n \bmos[n] \big)_\infty$ factors through $D$.
  If $H_{k(n)} = \Id - A_{N(k(n))}$, then the identity factors through $\Id - D$.
\end{proof}

We now combine theorems~\ref{thm:product-diag-bmo} and~\ref{thm:diagonal-bmo} and derive Theorem~\ref{thm:bmo-primary}.
\begin{proof}[Proof of Theorem~\ref{thm:bmo-primary}]
  By Wojtaszczyk's isomorphism, see~\cite{wojtaszczyk:1979}, the Banach space
  $\bmos$ is isomorphic to the infinite sum of its finite dimensional building blocks
  $\big( \sum_n \bmos[n] \big)_\infty$.
  Hence, in Theorem~\ref{thm:bmo-primary} we replace operators on $\bmos$ by operators on
  $\big( \sum_n \bmos[n] \big)_\infty$.
  Moreover, by Theorem~\ref{thm:product-diag-bmo}, it suffices to consider only \textbf{diagonal}
  operators on $\big( \sum_n \bmos[n] \big)_\infty$.
  In Theorem~\ref{thm:diagonal-bmo} we proved that for any diagonal operator $D$ on
  $\big( \sum_n \bmos[n] \big)_\infty$ the identity factors through $H = D$ or $H = \Id - D$, that
  is
  \begin{equation*}
    \vcxymatrix{\big( \sum_n \bmos[n] \big)_\infty \ar[r]^\Id \ar[d]_E
        & \big( \sum_n \bmos[n] \big)_\infty\\
        \big( \sum_n \bmos[n] \big)_\infty \ar[r]_H
        & \big( \sum_n \bmos[n] \big)_\infty \ar[u]_P}
      \qquad \|E\|\|P\| \leq C.
  \end{equation*}
  for some universal constant $C > 0$.
\end{proof}
%%%%%%%%%%%%%%%% END FACTORIZATION %%%%%%%%%%%%%%%%

\bibliographystyle{abbrv}
\bibliography{bibliography}
%\nocite{*}

\end{document}